\documentclass[11pt, reqno]{amsart}

\usepackage{amssymb}
\usepackage[shortlabels]{enumitem} \setlist[enumerate]{(a), leftmargin=*}
\usepackage{etoolbox}
\usepackage{eucal}
\usepackage[margin=1in]{geometry}
\usepackage[colorlinks=true, breaklinks=true, allcolors=blue]{hyperref}
\usepackage[dvipsnames]{xcolor}

\allowdisplaybreaks
\numberwithin{equation}{section}

\DeclareMathAlphabet{\mathsfit}{T1}{\sfdefault}{\mddefault}{\sldefault}
\SetMathAlphabet{\mathsfit}{bold}{T1}{\sfdefault}{}{\sldefault}

\renewcommand{\bar}[1]{\overline{#1}}
\newcommand\bLambda{\mathbf{\Lambda}}
\newcommand\bK{\mathbf{K}}
\newcommand\bU{\mathbf{U}}
\newcommand\bV{\mathbf{V}}
\newcommand\C{\mathbb{C}}
\newcommand\ev{\mathrm{ev}}
\newcommand\fg{\mathfrak{g}}

\newcommand\fh{\mathfrak{h}}
\newcommand\fn{\mathfrak{n}}
\newcommand\fa{\mathfrak{a}}
\newcommand\fb{\mathfrak{b}}
\newcommand\fsl{\mathfrak{sl}}
\newcommand\fosp{\mathfrak{osp}}
\newcommand\fgl{\mathfrak{gl}}
\newcommand\fz{\mathfrak{z}}
\newcommand\id{\mathrm{id}}
\newcommand{\kteta}{{\mathsf{k}_\Theta}}
\newcommand\sm{\mathsf{m}}
\newcommand\sk{\mathsf{k}}
\newcommand\sL{\mathsfit{L}}
\newcommand\sM{\mathsfit{M}}
\newcommand\sN{\mathsfit{N}}
\newcommand\sP{\mathsfit{P}}
\newcommand\sV{\mathsfit{V}}

\newcommand\Z{\mathbb{Z}}

\DeclareMathOperator{\Ann}{Ann}
\DeclareMathOperator{\ch}{ch}
\DeclareMathOperator{\Ext}{Ext}
\DeclareMathOperator{\HH}{H}
\DeclareMathOperator{\Hom}{Hom}

\DeclareMathOperator{\MaxSpec}{MaxSpec}
\DeclareMathOperator{\Spec}{Spec}
\DeclareMathOperator{\sdim}{sdim}
\DeclareMathOperator{\sch}{sch}
\DeclareMathOperator{\Supp}{Supp}

\theoremstyle{plain}
\newtheorem{theo}{Theorem}[section]
\newtheorem*{theo*}{Theorem}
\newtheorem{prop}[theo]{Proposition}
\newtheorem{lem}[theo]{Lemma}
\newtheorem{cor}[theo]{Corollary}

\theoremstyle{definition}

\newtheorem*{rem*}{Remark}
\newtheorem{rem}[theo]{Remark}

\newtoggle{details}
\iftoggle{details}{
\newcommand{\details}[1]{{\color{OliveGreen}\noindent\textbf{Details:}{#1}}}
                  }{
                  \newcommand{\details}[1]{}
                  }

\begin{document}
%

\title{Finite-dimensional representations of map superalgebras}

\author{Lucas Calixto}
\address{Department of Mathematics\\
		   Federal University of Minas Gerais\\
		   Belo Horizonte \\
		   Brazil}
		   \email{lhcalixto@ufmg.br}	
\author{Tiago Macedo}
\address{Instituto de Ci\^encia e Tecnologia \\
		  Universidade Federal de S\~ao Paulo \\
		  S\~ao Jos\'e dos Campos \\
		  Brazil}
		  \email{tmacedo@unifesp.br}



\begin{abstract}
We obtain a complete classification of all finite-dimensional irreducible modules over classical map superalgebras, provide formulas for their (super)characters and a description of their extension groups. Furthermore, we describe the block decomposition of the category of finite-dimensional modules for such map superalgebras.  As an application, we specialize our results to the case of loop superalgebras in order to obtain a classification of finite-dimensional irreducible modules and block decomposition of the category of finite-dimensional modules over affine Lie superalgebras.
\end{abstract}

\maketitle

\thispagestyle{empty}

\tableofcontents

\section*{Introduction} 
\label{sec:intro}

Given a finite-dimensional Lie superalgebra $\fg$ and an affine scheme of finite type $X$, both defined over $\C$, if one also considers $\fg$ as an affine space, then the set $M(X, \fg)$, of regular maps from $X$ to $\fg$, inherits a structure of Lie superalgebra from $\fg$.  These Lie superalgebras $M(X,\fg)$ are known as \emph{map superalgebras} and they admit an algebraic realization. Namely, there is an isomorphism of Lie superalgebras between $M(X,\fg)$ and the Lie superalgebra whose underlying vector space is the tensor product $\fg\otimes_\C \mathcal O(X)$, the $\Z_2$-grading is given by $(\fg \otimes_\C \mathcal O(X))_{\bar0} := \fg_{\bar0} \otimes_\C \mathcal O(X)$, $(\fg \otimes_\C \mathcal O(X))_{\bar1} := \fg_{\bar1} \otimes_\C \mathcal O(X)$, and the Lie superbracket is given by bilinearly extending
    \[
    [x \otimes a, \, y \otimes b]
    = [x, y] \otimes ab,
    \qquad \textup{for all } \ x, y \in \fg, \ a, b \in \mathcal O(X).
    \]
In this paper, we denote the map superalgebra $\fg \otimes_\C \mathcal O(X)$ by $\fg[\mathcal O(X)]$.

Map superalgebras encompass several superalgebras that, over the years, have been the subject of research.  Amongst those, current and loop superalgebras, which correspond respectively to the affine line ($X = \Spec \C[t]$) and the circle ($X = \Spec \C[t, t^{-1}]$), are particularly important to us.  They are closely related to affine superalgebras, which have applications to number theory, string theory, conformal (quantum) field theory (see, for instance, \cite{Kac78b, KW94, KW01, GO86, GOW87}), and are defined as follows.

Given a basic classical Lie superalgebra $\fg$ endowed with a non-degenerate, invariant, even, supersymmetric, bilinear form $(\cdot, \cdot)$, the affine superalgebra $\hat\fg$ contains an element $c$ such that: as a vector space, $\hat \fg = \fg[\C[t, t^{-1}]] \oplus \C c$; the $\Z_2$-grading on $\hat\fg$ is given by $\hat\fg_{\bar0} = \fg[\C[t, t^{-1}]]_{\bar0} \oplus \C c$, $\hat\fg_{\bar1} = \fg[\C[t, t^{-1}]]_{\bar1}$; and the Lie superbracket is given by bilinearly extending
\begin{equation} \label{eq:aff.sup}
[x \otimes t^n + \lambda c, \, y \otimes t^m + \mu c]
= [x,y] \otimes t^{n+m} + \delta_{n, -m} n (x,y) c,
\end{equation}
for all $x, y \in \fg$, $n, m \in \Z$, $\lambda, \mu \in \C$.  (Caveat: affine Lie superalgebras, as defined above, may be considered to be the derived subalgebras of affine Kac-Moody superalgebras, in some other papers.)

If $V$ is a finite-dimensional irreducible module for an affine Lie superalgebra $\hat\fg$, then $c \cdot V = 0$.  Hence, $V$ descends to a $\fg[\C[t, t^{-1}]]$-module, which continues to be irreducible.  Conversely, every finite-dimensional irreducible $\fg[\C[t, t^{-1}]]$-module can be inflated to a $\hat\fg$-module and the resulting $\hat\fg$-module remains irreducible.  Thus, the problem of classifying finite-dimensional irreducible modules for affine Lie superalgebras is equivalent to that for loop superalgebras.  In this paper, we approach this problem from the point-of-view of map superalgebras.

Finite-dimensional irreducible modules for classical basic map superalgebras (which are not map algebras) were classified by A.~Savage in \cite{Savage14}, L.~Calixto, A.~Moura, A.~Savage in \cite{CMS16}, and the authors in \cite{CM19}.  This classification is given in terms of evaluation and generalized evaluation modules.  Evaluation modules are as well-understood as finite-dimensional irreducible modules for finite-dimensional simple Lie superalgebras are.  However, this is not the case for gen\-er\-al\-ized evaluation modules.

The first goal of this paper is to describe all irreducible generalized evaluation modules. Using our description, we present formulas for their (super)characters and describe their extensions.  As applications, we describe the block decomposition of the categories of finite-dimensional modules for classical map and affine superalgebras.  We also establish an explicit relation between the highest weights of a finite-dimensional irreducible module relative to distinct choices of Borel subalgebras.  We now describe our main results in more detail.

Let $A$ be an associative, commutative, finitely-generated, unital $\C$-algebra (that is, $A = \mathcal O(X)$ for an affine scheme of finite type $X = \Spec(A)$), $n, k_1, \dotsc, k_n \in \Z_{>0}$, and $\sm_1, \dotsc, \sm_n$ be maximal ideals of $A$.  For each $i \in \{1, \dotsc, n\}$, let $\rho_i \colon \fg[A/\sm_i^{k_i}] \to \fgl(V_i)$ be a representation of $\fg[A/\sm_i^{k_i}]$ on $V_i$, and $\ev_{\sm_i^{k_i}} \colon A \twoheadrightarrow A/\sm_i^{k_i}$ be the canonical projection.  Define a $\fg[A]$-module structure on $V_1 \otimes \dotsm \otimes V_n$ by setting its associated representation to be
    \[
    \rho_1 \circ (\id_{\fg} \otimes \ev_{\sm_1^{k_1}}) \otimes \dotsm \otimes \rho_n \circ (\id_{\fg} \otimes \ev_{\sm_n^{k_n}}) \colon \fg[A] \to \fgl(V_1 \otimes \dotsm \otimes V_n).
    \]
These $\fg[A]$-modules will be denoted by $V_1^{\sm_1^{k_1}} \otimes \dotsm \otimes V_n^{\sm_n^{k_n}}$. Such modules will be called \emph{evaluation modules} when $k_i=1$ for all $i \in \{1, \dotsc, n\}$, and \emph{generalized evaluation modules} otherwise.

Every finite-dimensional irreducible $\fg[A]$-module is isomorphic to a tensor product of evaluation and generalized evaluation modules (see \cite{CFK10, NSS12, Savage14, CMS16, CM19}). Moreover, when $\fg_{\bar 0}$ is a semisimple Lie algebra, then every finite-dimensional irreducible $\fg[A]$-module is isomorphic to an evaluation module.  This is not true if $\fg_{\bar 0}$ is a reductive non-semisimple Lie algebra and, in particular, if $\fg$ is a basic classical Lie superalgebra of type~I.

If $\fg$ is a basic classical Lie superalgebra of type I, it is known that $\fg$ admits a consistent $\Z$-grading, $\fg = \fg_{-1} \oplus \fg_0 \oplus \fg_1$, where $\fg_{\pm 1}$ are irreducible $\fg_0$-modules, $\fg_0 = \fg'_0 \oplus \fz$ is a reductive Lie algebra with semisimple component $\fg'_0$ and center $\fz := \fg_0/\fg'_0$ \cite[\S1]{Kac78}. Let $\fg_\pm := \fg_0 \oplus \fg_{\pm1}$.  Given a quotient algebra $B$ of $A$ and a $\fg_0[B]$-module $\sM$, consider it as a $\fg_+[B]$-module by letting $\fg_1[B]\sM=0$.  The induced $\fg[B]$-module
	\begin{equation} \label{def:kac-like}
    \bK_B(\sM) := \bU(\fg[B])\otimes_{\bU(\fg_{_+}[B])} \sM
	\end{equation}
is naturally a $\fg[A]$-module via the projection from $\fg[A]$ to $\fg[B]$. We call this module a \emph{Kac-like module}. One can check that, if $\sL$ is a finite-dimensional irreducible $\fg_0[B]$-module, then $\bK_B(\sL)$ admits a unique irreducible quotient, which will be denoted by $\bV_B(\sL)$.  Moreover, it follows from \cite[Theorem~7.2]{Savage14} that every finite-dimensional irreducible $\fg[A]$-module is isomorphic to $\bV_A(\sL)$, for some finite-dimensional irreducible $\fg_0[A]$-module $\sL$.

In Section~\ref{sec:class}, we describe all generalized evaluation modules in terms of tensor products of certain Kac-like modules. This enables us to classify all irreducible finite-dimensional $\fg[A]$-modules. Our fist main result is the following (see Theorem~\ref{theo:class.irreps}):

\begin{theo*}
If $L$ is a finite-dimensional irreducible $\fg[A]$-module, then $L$ is isomorphic to a tensor product of Kac-like modules and evaluation modules.
\end{theo*}

By using this description we obtain formulas for (super)characters of every finite-dimensional irreducible $\fg[A]$-module (see Proposition~\ref{prop:dim+char}).  Notice further that, specializing these results to the case of loop superalgebras, one obtains a classification of finite-dimensional irreducible modules for affine Lie superalgebras associated to basic classical Lie superalgebras.

In Section~\ref{sec:exts}, we use Theorem~\ref{theo:class.irreps}, together with some cohomological methods (the most important one being the Lyndon-Hochschild-Serre spectral sequence), to compute extensions between finite-dimensional irreducible $\fg[A]$-modules. It has been proven in \cite[Theorem~5.7]{CM19} that it is enough to consider extensions between (generalized) evaluation modules evaluated at the same maximal ideal; the general case is obtained from this one via direct sums. Moreover, extensions between two evaluation modules were computed in \cite[Example~5.8]{CM19}. To tackle the remaining cases, in Theorem~\ref{thm:ext.kac} we prove our second main result, which is the following:

\begin{theo*}
Let $\sV_1, \sV_2$ be finite-dimensional irreducible $\fg'_0$-modules, let $\sm$ be a maximal ideal of $A$, let $\Theta_1, \Theta_2 \in \fz[A]^*$, assume that $\Theta_1(\fz \otimes \sm^n) = \Theta_2(\fz \otimes \sm^n) = 0$ for some $n \in \Z_{>0}$ and that either $\Theta_1(\fz \otimes \sm) \ne 0$ or $\Theta_2(\fz \otimes \sm) \ne 0$.
\begin{enumerate}
\item
If $\Theta_1 \ne \Theta_2$ and $\Theta_1(z) - \Theta_2(z) \in \Z_{\ge0}$, then
    \[
    \Ext^1_{\fg[A]} \left( \bV_A(\Theta_1 \boxtimes \sV_1^\sm),\, \bV_A(\Theta_2 \boxtimes \sV_2^\sm) \right)
    \cong \Hom_{\fg_0[A]} (\fg_{-1}[\sk_{\Theta_1}] \otimes (\Theta_1 \boxtimes \sV_1^\sm),\, \Theta_2 \boxtimes \sV_2^\sm).
    \]

\item
If $\Theta_1 \ne \Theta_2$ and $\Theta_2(z) - \Theta_1(z) \in \Z_{\ge0}$, then
    \begin{align*}
    \Ext^1_{\fg[A]} \left( \bV_A(\Theta_1 \boxtimes \sV_1^\sm),\, \bV_A(\Theta_2 \boxtimes \sV_2^\sm) \right)
    \cong {}&{} \Hom_{\fg_0[A]} (\fg_{-1}[\sk_{\Theta_2}] \otimes (\Theta_2 \boxtimes \sV_2^\sm),\, \Theta_1 \boxtimes \sV_1^\sm).
    \end{align*}
    
\item
If $\Theta_1(z) - \Theta_2(z) \notin \Z$, then $\Ext^1_{\fg[A]} \left( \bV_A(\Theta_1 \boxtimes \sV_1^\sm), \bV_A(\Theta_2 \boxtimes \sV_2^\sm) \right) = 0$.

\item
In the case where $\Theta_1 = \Theta_2$, we have $\Ext^1_{\fg[A]} \left( \bV_A(\Theta_1 \boxtimes \sV_1^\sm), \, \bV_A(\Theta_2 \boxtimes \sV_2^\sm) \right) = 0$ if and only if $\Ext^1_{\fg_0[A]} \left( \Theta_1 \boxtimes \sV_1^\sm, \, \Theta_2 \boxtimes \sV_2^\sm \right) = 0$.
\end{enumerate}
\end{theo*}

In Section~\ref{sec:apps}, we present two applications of our results, both regarding general classical Lie superalgebra (not just those of type~I).  First, using Theorem~\ref{thm:ext.kac} and \cite[Theorem~5.7]{CM19}, we describe the block decomposition of the categories of finite-dimensional $\fg[A]$ and $\hat\fg$-modules in terms of spectral characters (see Propositions~\ref{prop:blk.map} and \ref{prop:blk.aff}).  Then, we use Proposition~\ref{prop:main} to establish a relation between the highest weights of a given finite-dimensional irreducible $\fg[A]$-module with respect to non-conjugate Borel subalgebras of $\fg[A]$ (see Proposition~\ref{prop:chang.Borel}).  Recall that, unlike the Lie algebra case, for a finite-dimensional simple Lie superalgebra, not all Borel subalgebras are conjugated via the Weyl group of $\fg_{\bar0}$.  In fact, to obtain Borel subalgebras in distinct orbits with respect to the action of the Weyl group, one needs to use the so-called odd reflections (see \cite{Ser11} for details).  This is an important distinction between the super and non-super cases and has deep consequences in the study of highest-weight modules for Lie superalgebras.

We point out that our results are new for all $A \ne \C$, including the current and loop cases

\subsection*{Notation} Throughout this paper, the ground field will be $\C$.  All vector spaces and tensor products will be considered to be over $\C$, unless otherwise stated. For any Lie superalgebra $\fa$, let $\bU(\fa)$ denote its universal enveloping superalgebra, and for any vector space $V$ let $\bLambda (V)$ denote its Grassmann algebra.

\section{Preliminaries}  
\label{sec:prelim}

Let $\fg \cong \fsl(m|n)$ or $\fosp(2|2n)$, $1 \le m \le n$, $1 < n$, that is, that $\fg$ is either a basic classical Lie superalgebra of type~I or a finite-dimensional Lie superalgebra isomorphic to $\fsl(n|n)$.  In this section, we fix some notation and recall a few results that will be used throughout the paper.

Fix a Cartan subalgebra $\fh \subseteq \fg$, which is, by definition, a Cartan subalgebra of the reductive Lie algebra $\fg_0$.  Under the adjoint action of $\fh$, we have a root space decomposition:
    \[
    \fg = \fh \oplus \bigoplus_{\alpha \in \fh^* \setminus \{0\}} \fg_{\alpha},
    \quad \textup{where} \quad
    \fg_\alpha := \{ x \in \fg \mid [h, x] = \alpha(h) x \textup{ for all } h \in \fh \}.
    \]
Denote by $R$ the set of roots, $\{ \alpha \in \fh^* \setminus \{0\} \mid \fg_\alpha \neq \{ 0 \} \}$.  For each $\alpha \in R$, the root space $\fg_\alpha$ is either purely even, that is, $\fg_\alpha \subseteq \fg_0$, or $\fg_\alpha$ is purely odd, that is, $\fg_\alpha \subseteq \fg_{\bar1}$.  Let $R_0 = \{ \alpha \in R \mid \fg_\alpha \subseteq \fg_0 \}$ be the set of even roots and $R_{\bar1} = \{ \alpha \in R \mid \fg_\alpha \subseteq \fg_{\bar1} \}$ be the set of odd roots.  Every choice of a set of simple roots $\Delta \subseteq R$ yields a decomposition $R = R^+ \sqcup R^-$, where $R^+ := \Z_{\ge0}\Delta \cap R$ (resp. $R^- := \Z_{\le0}\Delta \cap R$) denotes the set of positive (resp. negative) roots.  Define
    \[
    \Delta_0 = \Delta \cap R_0, \quad
    \Delta_{\bar1} = \Delta \cap R_{\bar1}, \quad
    R^\pm_0 = R_0 \cap R^\pm
    \quad \textup{and} \quad
    R^\pm_{\bar1} = R_{\bar1} \cap R^\pm.
    \]
A choice of simple roots $\Delta \subseteq R$ also induces a triangular decomposition $\fg = \fn^- \oplus \fh \oplus \fn^+$, where $\fn^\pm := \bigoplus_{\alpha \in R^\pm} \fg_\alpha$; and for every choice of triangular decomposition $\fg_0 = \fn_0^- \oplus \fh \oplus \fn_0^+$, one can choose $\Delta$ in such a way that $\fn^\pm = \fn^\pm_0 \oplus \fg_{\pm1}$.

Choose elements $z$ and, for each $\alpha \in R^+$, $y_\alpha, h_\alpha, x_\alpha$, in such a way that: $\{ z \}$ is a basis of $\fz$; for each $\alpha \in R^+$, $\{ y_\alpha \}$ is a basis of $\fg_{-\alpha}$, $\{ x_\alpha \}$ is a basis of $\fg_\alpha$, $[x_\alpha, y_\alpha] = h_\alpha \in \fh$; and the set $\{y_\alpha,\, h_i,\, x_\alpha \mid \alpha \in R^+_0,\, i \in \Delta_0 \}$ is a Chevalley basis of $\fg'_0$.  In particular:
    \begin{equation} \label{eq:z.hbeta.ne0}
    \textup{for every } \beta \in R^+_{\bar1}, \ \ 
    z = c_\beta h_\beta + \sum_{i \in \Delta_0} c_i h_i \ \ \ 
    \textup{for some } \{c_i \mid i \in \Delta_0 \} \subseteq \C \textup{ and } c_\beta \in \C \setminus \{0\}.
    \end{equation}
Using this basis, define a linear map $\tau : \fg[A] \to \fg[A]$ by linearly extending
    \[
    \tau(x_\alpha\otimes a) = y_\alpha\otimes a, \
    \tau(y_\alpha\otimes a) = x_\alpha\otimes a, \
    \tau(h\otimes a) = h\otimes a, \quad
    \textup{ for all } \ \alpha \in R^+,\, h \in \fh,\, a\in A.
    \]
One can check that $\tau$ defines an anti-involution of $\fg[A]$, \details{ that is, $\tau : \fg[A] \to \fg[A]$ is a linear map, such that $\tau[x,y] = [\tau(y), \tau(x)]$, for all $x,y\in \fg[A]$, and $\tau^2 = {\rm id}_{\fg[A]}$,} such that $\left. \tau \right|_{\fh[A]} = {\rm id}_{\fh[A]}$.  Given a finite-dimensional $\fg[A]$-module $M$, denote by $M^\vee$ the $\fg[A]$-module whose underlying vector space is $M^*$ and the action of $\fg[A]$ is given by:
    \begin{equation} \label{eq:vee.action}
    (x f)(m) := f(\tau(x)m) \qquad
    \textup{for all } x \in \fg[A],\ f \in M^*,\ m \in M.
    \end{equation}
\details{
Since
    \begin{align*}
    (x y f)(m) -(-1)^{|x||y|} (y x f)(m)
    &= (yf)(\tau(x)m) -(-1)^{|x||y|} (xf)(\tau(y)m) \\
    &= f(\tau(y)\tau(x)m) -(-1)^{|x||y|} f(\tau(x)\tau(y)m) \\
    &= f((\tau(y)\tau(x) -(-1)^{|y||x|} \tau(x)\tau(y))m) \\
    &= f([\tau(y), \tau(x)] m) \\
    &= f(\tau[x,y]m) \\
    &= ([x,y]f) (m),
    \end{align*}
for all homogeneous $x, y \in \fg[A]$, $f \in M^*$, $m \in M$, the action defined in \eqref{eq:vee.action} in fact endows $M^\vee$ with a $\fg[A]$-module structure.
}

The following result will be used in the proof of Theorem~\ref{thm:ext.kac}.

\begin{lem} \label{lem:dual.ext}
In the category of finite-dimensional $\fg[A]$-modules, $M \mapsto M^\vee$ defines an exact contravariant functor such that $L^\vee \cong L$ for all finite-dimensional irreducible $\fg[A]$-module $L$.
\end{lem}
\begin{proof}
At the level of vector spaces, $^\vee$ is the same as the usual dual $^*$.  This implies that $M \mapsto M^\vee$ defines an exact contravariant functor in the category of finite-dimensional $\fg[A]$-modules.  The fact that $L^\vee \cong L$ for all finite-dimensional irreducible $\fg[A]$-module $L$ follows from \cite[Lemma~4.5]{Savage14}.
\iftoggle{details}{

}{\qedhere}
\details{
Let $L$ is a finite-dimensional irreducible $\fg[A]$-module.  If $N$ is a $\fg[A]$-submodule of $L^\vee$, then we have a short exact sequence of finite-dimensional $\fg[A]$-modules: $0 \to N \to L^\vee \to L^\vee/N \to 0$.  Since $^\vee$ is a contravariant exact functor, then $0 \to (L^\vee / N)^\vee \to (L^\vee)^\vee \to N^\vee \to 0$ is also an exact sequence of $\fg[A]$-modules.  Now, since $\tau^2 = {\rm id}_{\fg[A]}$ and $L$ is finite dimensional, then $(L^\vee)^\vee \cong L$.  Moreover, since $L$ is irreducible, then: either $(L^\vee / N)^\vee = 0$, or $(L^\vee / N)^\vee \cong L$.  Notice that, for every finite-dimensional $\fg[A]$-module, we have that $\dim M^\vee = \dim M$.  Hence, in the first case, we have that $\dim N = \dim L^\vee$, which implies that $N = L^\vee$; and in the second case, we have that $\dim N = 0$.  Either way, $N$ is not a proper submodule of $L^\vee$.

In order to finish the proof, we need to show that $L^\vee \cong L$ for all finite-dimensional irreducible $\fg[A]$-module $L$.  By \cite[Lemma~4.5]{Savage14}, there exists $\psi \in \fh[A]^*$ and $v_\psi \in L$ such that: $\fn^+[A]v_\psi=0$, $(h\otimes a)v_\psi = \psi(h\otimes a)v_\psi$ for every $h \in \fh$, $a \in A$, and $L = \bU(\fn^-[A])v_\psi$.  Moreover, $L$ can be uniquely determined by such $\psi$.  Thus, it is enough to show that there exists $f \in L^\vee$ such that: $\fn^+[A] f = 0$, $(h \otimes a) f = \psi(h\otimes a) f$ for every $h \in \fh$, $a \in A$.

Since $L$ is finite dimensional, there exist $w_2, \dotsc, w_n \in L$ such that $\{v_\psi, w_2, \dotsc, w_n\}$ is a basis for $L$.  Consider the unique linear functional $f \in L^\vee$ that satisfies: $f(v_\psi) = 1$ and $f(w_i) = 0$ for all $i \in \{2, \dotsc, n\}$.  Thus, for all $\alpha \in R^+$, $h \in \fh$, $a \in A$, $i \in \{2, \dotsc, n\}$, we have:
    \begin{gather*}
    ((x_\alpha\otimes a)f)(v_\psi) = f((y_\alpha\otimes a)v_\psi) = 0, \qquad
    ((x_\alpha\otimes a)f))(w_i) = f((y_\alpha\otimes a)w_i) = 0, \\
    ((h\otimes a)f)(v_\psi) = f((h\otimes a)v_\psi) = \psi(h\otimes a), \qquad
    ((h\otimes a)f)(w_i) = f((h\otimes a)w_i) = 0.
    \end{gather*}
This shows that $\fn^+[A]f = 0$ and $(h\otimes a)f = \psi(h\otimes a)f$ for every $h \in \fh$, $a \in A$.
}
\end{proof}

We finish this section by establishing a few commutation relations that will be used in the proof of Proposition~\ref{prop:main}.  Their proofs are straight-forward.

\begin{lem} \label{lem:comm.rels}
Let $k \in \Z_{>0}$ and $a_0, a_1, \dotsc, a_k \in A$.
\begin{enumerate}
\item For all $\alpha \in R^+_{\bar 1}$, we have:
\begin{align*}
(x_\alpha \otimes a_0){}&{}(y_\alpha \otimes a_1) \dotsm (y_\alpha \otimes a_k) \\
={}&{}(-1)^k (y_\alpha \otimes a_1) \dotsm (y_\alpha \otimes a_k) (x_\alpha \otimes a_0)\\
&{}+ \sum_{j=1}^k (-1)^{j-1} (y_\alpha \otimes a_1) \dotsm (y_\alpha \otimes a_{j-1}) (y_\alpha \otimes a_{j+1}) \dotsm (y_\alpha \otimes a_k) (h_\alpha \otimes a_0a_j).
\end{align*}

\item For all $\alpha, \beta \in R^+_{\bar1}$, we have:
\begin{align*}
(x_\alpha \otimes a_0){}&{}(y_\beta \otimes a_1) \dotsm (y_\beta \otimes a_k) \\
={}&{}(-1)^k (y_\beta \otimes a_1) \dotsm (y_\beta \otimes a_k) (x_\alpha \otimes a_0)\\
&{}+ \sum_{j=1}^k (-1)^{j-1} (y_\beta \otimes a_1) \dotsm (y_\beta \otimes a_{j-1})  ([x_\alpha, y_\beta] \otimes a_0a_j) (y_\beta \otimes a_{j+1}) \dotsm (y_\beta \otimes a_k),
\end{align*}
where $[x_\alpha, y_\beta] \in \fg_0$.

\item For all $\beta \in R^+_{\bar1}$, $\gamma \in R^+_0$, we have:
\begin{align*}
(x_\gamma \otimes a_0){}&{}(y_\beta \otimes a_1) \dotsm (y_\beta \otimes a_k) \\
={}&{} (y_\beta \otimes a_1) \dotsm (y_\beta \otimes a_k) (x_\gamma \otimes a_0)\\
&{}+ \sum_{j=1}^k (-1)^{k-j} (y_\beta \otimes a_1) \dotsm (y_\beta \otimes a_{j-1}) (y_\beta \otimes a_{j+1}) \dotsm (y_\beta \otimes a_k) ([x_\gamma, y_\beta] \otimes a_0a_j),
\end{align*}
where $[x_\gamma, y_\beta] \in \fg_{-1}$.
\end{enumerate}
\end{lem}
\iftoggle{details}{{\color{OliveGreen}{
\begin{proof}
\begin{enumerate}
\item Since $\alpha$ is an odd root, then $x_\alpha y_\alpha = - y_\alpha x_\alpha + h_\alpha$ and $h_\alpha y_\alpha = y_\alpha h_\alpha$.  The result follows from successively applying these two relations.

\item Since $\alpha, \beta$ are odd roots, then $x_\alpha y_\beta = - y_\beta x_\alpha + [x_\alpha, y_\beta]$, where $[x_\alpha, y_\beta] \in [\fg_1, \fg_{-1}] \subseteq \fg_0$.  The result follows from successively applying this relation.

\item Since $\beta$ is an odd root and $\gamma$ is an even root, then $x_\gamma y_\beta = y_\beta x_\gamma + [x_\gamma, y_\beta]$, where $[x_\gamma, y_\beta] \in [\fg_0, \fg_{-1}] \subseteq \fg_{-1}$.  Hence $[[x_\gamma, y_\beta], y_\beta] \in [\fg_{-1}, \fg_{-1}] = 0$.  This implies that $[x_\gamma, y_\beta]y_\beta = - y_\beta [x_\gamma, y_\beta]$.  The result follows from successively applying these relations.
\qedhere
\end{enumerate}
\end{proof}}}}{}

\section{Finite-dimensional irreducible modules}  
\label{sec:class}

Assume that $\fg \cong \fsl(m|n)$ or $\fosp(2|2n)$, $1 \le m \le n$, $1 < n$, that is, that $\fg$ is a basic classical Lie superalgebra of type~I or a finite-dimensional Lie superalgebra isomorphic to $\fsl(n|n)$, and also that $A$ is an associative, commutative, finitely-generated, unital algebra over $\C$.

Let $\sm$ be a maximal ideal of $A$, $n > 1$, and $L$ be a finite-dimensional irreducible $\fg[A/\sm^n]$-module.  Then $L^{\sm^n}$ is an irreducible $\fg[A]$-module, on which $\fg \otimes \sm^n$ acts trivially.  Thus, $L^{\sm^n}$ is isomorphic to $\bV_A(\sM)$ for some finite-dimensional irreducible $\fg_0[A]$-module $\sM$ (see Introduction). Moreover, it follows from the proof of \cite[Proposition~6.4]{Savage14} that $L^{\sm^n}$ is isomorphic to $\bV_{A/I}(\sM)$ for any ideal $I \subseteq A$ such that $(\fg \otimes I)L^{\sm^n} = 0$.  We want to prove that, in fact, when such $I$ is largest, then $L^{\sm^n} \cong \bV_{A/I} (\sM) \cong \bK_{A/I} (\sM)$.  Since $\bV_{A/I} (\sM)$ is a quotient of $\bK_{A/I} (\sM)$, we begin by studying the submodules of $\bK_{A/I} (\sM)$.

\begin{lem} \label{lem:silly}
Let $I$ be a finite-codimensional ideal of $A$, $\sM$ be a finite-dimensional irreducible $\fg_0[A/I]$-module, and $d = \dim \fg_{-1}[A/I]$.  If $W \ne \{0\}$ is a $\fg[A/I]$-submodule of $\bK_{A/I} (\sM)$, then there exist a nonzero $p \in \bLambda^d (\fg_{-1}[A/I])$ and a highest-weight vector $v^+ \in \sM$, such that $p \otimes v^+ \in W$ is nonzero.
\end{lem}
\begin{proof}
Since $W \ne \{0\}$, there exists a nonzero vector $w \in W$.  Now, recall that, as a vector space, $\bK_{A/I}(\sM) \cong \bU(\fg_{-1}[A/I]) \otimes \sM$.  Thus, there exist $p_1, \dotsc, p_k \in \bU(\fg_{-1}[A/I])$ and $v_1, \dotsc, v_k \in V$ such that $w = \sum_{i = 1}^k p_i \otimes v_i$.  Further, recall that $\bU(\fg_{-1}[A/I]) \cong \bLambda  (\fg_{-1}[A/I])$, and consider the usual grading, $\bLambda  (\fg_{-1}[A/I]) = \bigoplus_{i=0}^{d} \bLambda^i (\fg_{-1}[A/I])$.  Without loss of generality, we can assume that $p_1 \in \bLambda^{d_1} (\fg_{-1}[A/I]), \dotsc, p_k \in \bLambda^{d_k} (\fg_{-1}[A/I])$, and $d_1 \le \dotsb \le d_k$.

Now, recall that $\bLambda^i (\fg_{-1}[A/I]) = 0$ for all $i > d$ and $\bLambda^d (\fg_{-1}[A/I])$ is 1-dimensional.  Hence there exists $p \ne 0$ such that $\bLambda^d (\fg_{-1}[A/I]) = \C p$.  Further, since $W$ is a $\fg[A/I]$-submodule of $\bK_{A/I}(\sM)$ and $d_1 \le \dotsb \le d_k$, then there exists $q \in \bLambda^{d - d_1} (\fg_{-1}[A/I])$ such that
    \[
    0\neq qw = p \otimes v' \in W \cap \left(\bLambda^d(\fg_{-1}[A/I])\otimes \sM\right).
    \]
Furthermore, since $\fg_{-1}$ is a $\fg_0$-module (via the adjoint representation), then for each $i \ge 0$, $\bLambda^i (\fg_{-1}[A/I])$ is a $\fg_0[A/I]$-submodule of $\bLambda(\fg_{-1}[A/I])$.  Moreover, since $\bLambda^d (\fg_{-1}[A/I])$ is 1-dimen\-sion\-al, then it is a trivial $\fg_0[A/I]$-module.  This implies that, for every $x \in \fg_0[A/I]$, $p \in \bLambda^d (\fg_{-1}[A/I])$ and $v \in \sM$, we have:
    \[
    x(p\otimes v) = p \otimes xv \in \bK_{A/I}(\sM).
    \]
Since $\sM$ is an irreducible $\fg_0[A/I]$-module, then there exists $u \in \bU(\fg_0[A/I])$ such that $uv'$ is a highest-weight vector of $\sM$.  Hence
    \[
    u(qw) = u (p \otimes v') = p \otimes uv' \in W \quad \textup{is nonzero}.
    \qedhere
    \]
\end{proof}

Let now $\Theta \in \fz[A]^*$, $\sm$ be a maximal ideal of $A$, assume that $\Theta(\fz \otimes \sm^l) = 0$ for some $l > 0$, let $n = \min \{ l > 0 \mid \Theta(\fz \otimes \sm^l) = 0 \}$, and assume that $n > 1$.  Let
    \begin{gather*}
    \mathcal K_\Theta 
    := \{ I \subset A \mid I \textup{ is a proper ideal of $A$ such that } \Theta(\fz \otimes I) = 0\}, \\
    \kteta \textup{ be the largest (unique, with respect to containment) ideal within } \mathcal K_\Theta,
    \end{gather*}
and notice that $\sm^n \subseteq \kteta \subsetneq \sm$.

Since $A$ is an associative, commutative, finitely-generated unital $\C$-algebra, by Hilbert's Nullstellensatz, there exist $r > 0$, $a_1, \dotsc, a_r \in \C$, and surjective homomorphisms of algebras
    \[
    \C[t_1, \dotsc, t_r] / \langle t_1-a_1, \dotsc, t_r-a_r \rangle^n
    \twoheadrightarrow A / \sm^n
    \twoheadrightarrow A / \kteta.
    \]
For each $i \in \{1, \dotsc, r\}$, denote $(t_i-a_i)$ by $T_i$, and for each $f \in \C[t_1, \dotsc, t_r]$, denote its image onto $A / \kteta$ by $\bar f$.  Let $\mathcal I_\Theta := \{ (i_1, \dotsc, i_r) \in \Z_{\ge0}^r \mid \Theta (z \otimes \overline{T_1}^{i_1} \dotsm \overline{T_r}^{i_r}) \ne 0 \}$.  Consider the usual partial order on $\Z^r$ \details{ $(i_1, \dotsc, i_r) \preceq (j_1, \dotsc, j_r)$ when $i_1 \le j_1 , \dotsc, i_r \le j_r$}.  Since $\Theta(\fz \otimes \sm^n) = 0$, then $\mathcal I_\Theta$ is a finite subset of $\Z^r$, and hence there exist elements in $\mathcal I_\Theta$ which are maximal with respect to this partial order. Denote the set of these maximal elements by $\widehat{\mathcal I}_\Theta$.  For each $\widehat n = (n_1, \dotsc, n_r) \in \widehat{\mathcal I}_\Theta$, $(i_1, \dotsc, i_r) \preceq \widehat n$, $\alpha \in R^+_{\bar1}$, let 
    \[
    \left( y_\alpha \otimes \overline{T_1}^{i_1} \dotsm \overline{T_r}^{i_r} \right)^\star
    := x_\alpha \otimes \overline{T_1}^{n_1-i_1} \dotsm \overline{T_r}^{n_r-i_r}
    \in \fg_1[A/\kteta],
    \]
and for each $\lambda \in \C$, $p_1, p_2 \in \bU(\fg_{-1}[A/\kteta])$, let $(\lambda p_1 + p_2)^\star = \lambda p_1^\star + p_2^\star$, $(p_1 p_2)^\star := p_2^\star p_1^\star$.  (Notice that the definition of $^\star$ depends on the choice of $\widehat n$, which may not be unique.  For the rest of this section, however, the respective choices of $\widehat n$ will be clear from the context.  Thus this notation shall cause no confusion.)

Notice that, for a maximal ideal $\sm \subseteq A$, a functional $\Theta \in \fz[A]^*$ and a $\fg'_0$-module $\sV$, the Kac-like module $\bK_{A/I}(\Theta \boxtimes \sV^\sm)$ is well-defined if and only if $I \in \mathcal K_\Theta$.  The following result determines when such Kac-like modules are irreducible.

\begin{prop} \label{prop:main}
Let $\sm \subseteq A$ be a maximal ideal, $\Theta \in \fz[A]^*$ be such that $\sm^n \subseteq \kteta \subsetneq \sm$ for some $n > 1$, and $\sV$ be a finite-dimensional irreducible $\fg'_0$-module.  The $\fg[A/I]$-module $\bK_{A/I} (\Theta \boxtimes \sV^\sm)$ is irreducible if and only if $I = \kteta$.
\end{prop}
\begin{proof}
We will begin by proving that the Kac-like module $\bK_{A/I}(\Theta \boxtimes \sV^\sm)$ is reducible if $I \in \mathcal K_{\Theta}$ is not $\kteta$.  Let $v \ne 0$ be a highest-weight vector in $\bK_{A/I}(\Theta \boxtimes \sV^\sm)$. Since $I \in \mathcal K_\Theta$ and $\kteta$ is the largest ideal in $\mathcal K_\Theta$, then $I \subseteq \kteta$.  Since $I \ne \kteta$, then there exists $a \in \kteta \setminus I$.  Thus, for a simple root $\alpha\in \Delta_{\bar1}$, the element $(y_\alpha \otimes (a + I)) \otimes v \in \bK_{A/I}(\Theta \boxtimes \sV^\sm)$ is nonzero.  Since $\alpha$ is a simple odd root and $a \in \kteta$, one can check that $(y_\alpha \otimes (a+I)) \otimes v$ is a highest-weight vector in $\bK_{A/I}(\Theta \boxtimes \sV^\sm)$.  In particular, $v$ is not in the submodule of $\bK_{A/I}(\Theta \boxtimes \sV^\sm)$ generated by $(y_\alpha \otimes (a+I)) \otimes v$.  This shows that $\bK_{A/I}(\Theta \boxtimes \sV^\sm)$ is reducible.

Next, we will prove that $\bK_{A/\kteta}(\Theta \boxtimes \sV^\sm)$ is irreducible. Let $d = \dim \fg_{-1}[A/\kteta]$ and choose a nonzero $p \in \bLambda^d (\fg_{-1}[A/\kteta])$.  For every highest-weight vector $v \in \Theta \boxtimes \sV^\sm$, and every $\widehat n \in \widehat{\mathcal I}_\Theta$, the vector $p^\star p \otimes v$ is a nonzero multiple of $v$.  In fact, this follows from successively applying Lemma~\ref{lem:comm.rels}, together with the fact that, for every $\beta \in R^+_{\bar1}$ and $\widehat n = (n_1, \dotsc, n_r) \in \widehat{\mathcal I}_{\Theta}$, there exists $\lambda \ne 0$ such that $(h_\beta \otimes \overline{T_1}^{n_1} \dotsm \overline{T_r}^{n_r}) v = \lambda \Theta(z \otimes \overline{T_1}^{n_1} \dotsm \overline{T_r}^{n_r}) v \ne 0$ (see \eqref{eq:z.hbeta.ne0}). Since, by Lemma~\ref{lem:silly}, any nonzero submodule of $\bK_{A/\kteta}(\Theta \boxtimes \sV^\sm)$ contains the vector $p\otimes v$, the result follows.
\end{proof}

\begin{rem} \label{rem:main}
For all ideals $J \subseteq I \subseteq A$, there is a canonical projection $\pi_{I,J,} : A/J \twoheadrightarrow A/I$ of algebras, which induces a surjective homomorphism $
\tilde\pi_{I,J} : \fg[A/J] \twoheadrightarrow \fg[A/I]$ of Lie algebras.  Notice that the pull-back of every irreducible $\fg[A/I]$-module through $\tilde\pi_{I,J}$ is an irreducible $\fg[A/J]$-module.  In particular, by Proposition~\ref{prop:main}, the pull-back of the Kac-like module $\bK_{A/\kteta}(\Theta \boxtimes \sV^\sm)$ through $\tilde\pi_{\kteta, \sm^n}$ is irreducible as a $\fg[A/\sm^n]$-module for all $n > 0$ such that $\sm^n \subseteq \kteta$.
\end{rem}

\begin{rem}
Notice that throughout this section, we have assumed that $n > 1$.  In the case where $n = 1$  \details{ and thus, $\kteta = \sm$}, we would have $A / \kteta \cong \C$ and $\fg[A/\kteta] \cong \fg$.  In this case, any irreducible finite-dimensional $\fg[A/\kteta]$-module is isomorphic to an evaluation module, and hence, such a module is isomorphic to a Kac module if and only if its highest weight is \emph{typical} \cite[Theorem~1]{Kac78}.
\end{rem}

The following result provides a complete classification of all finite-dimensional irreducible $\fg[A]$-modules.  Compare it with \cite[Theorem~7.2]{Savage14}, \cite[Theorem~7.1]{CMS16} and \cite[Theorem~3.9]{CM19}.

\begin{theo} \label{theo:class.irreps}
If $L$ is a finite-dimensional irreducible $\fg[A]$-module, then there exist a finite\--dimen\-sion\-al evaluation $\fg[A]$-module $M$, finite-dimensional irreducible $\fg_0'$-modules $\sV_1, \dotsc, \sV_l$, distinct maximal ideals $\sm_1, \dotsc, \sm_l \subseteq A$, integers $n_1, \dotsc, n_l > 1$, and $\Theta_1, \dotsc, \Theta_l \in \fz[A]^*$, such that $\sm_i^{n_i} \subseteq \sk_{\Theta_i} \subsetneq \sm_i$ for all $i \in \{1, \dotsc, l\}$, and
    \begin{equation} \label{eq:class.irred}
    L \cong (\bK_{A/\sk_{\Theta_1}}(\Theta_1 \boxtimes \sV_1^{\sm_1}))^{\sm_1^{n_1}} \otimes \dotsb \otimes (\bK_{A/\sk_{\Theta_l}} (\Theta_l \boxtimes \sV_l^{\sm_l}))^{\sm_l^{n_l}} \otimes M,
    \end{equation}
as $\fg[A]$-modules.
\end{theo}
\begin{proof}
Since $L$ is a finite-dimensional irreducible $\fg[A]$-module, by \cite[Lemma~6.3(c)]{Savage14}, it is a generalized evaluation module, that is,
    \[
    L \cong W_1^{\sm_1^{n_1}} \otimes \dotsb \otimes W_1^{\sm_k^{n_k}},
    \]
for some $k, n_1, \dotsc, n_k \in \Z_{>0}$, distinct maximal ideals $\sm_1, \dotsc, \sm_k \subseteq A$, and irreducible $\fg[A/\sm_i^{n_i}]$-modules $W_i$ ($i \in \{1, \dotsc, k\}$).  Fix $i \in \{1, \dotsc, k\}$.  If $n_i = 1$, then $W_i^{\sm_i}$ is an evaluation module.  Otherwise, the module $W_i$ is irreducible for $\fg[A/\sm_i^{n_i}]$, and it follows from Proposition~\ref{prop:main} and Remark~\ref{rem:main} that $W_i \cong \bK_{A/\sk_{\Theta_i}} (\Theta_i \boxtimes \sV_i^{\sm_i})$ for some finite-dimensional irreducible $\fg_0'$-module $\sV_i$ and some $\Theta_i \in \fz[A]^*$ such that $\sm_i^{n_i} \subseteq \sk_{\Theta_i} \subsetneq \sm_i$.
\end{proof}

\begin{rem}
Let $\Pi$ denote the parity change functor in the category of super vector spaces. \details{ $\Pi$ preserves the morphisms and, for each super vector space $V = V_{\bar 0} \oplus V_{\bar 1}$, $\Pi V$ is defined to be the vector space $V$ with $\Z_2$-grading given by $(\Pi V)_{\bar 0} = V_{\bar 1}$ and $(\Pi V)_{\bar 1} = V_{\bar 0}$.}  Notice that $\Pi$ induces a functor in the category of finite-dimensional $\fg[A]$-modules.  Also notice that, since we assume that the morphisms in the category of $\fg[A]$-modules are even, then for every $\fg[A]$-module $M$, the $\fg[A]$-modules $M$ and $\Pi M$ are not isomorphic. Thus, the classification provided in Theorem~\ref{theo:class.irreps}, as well as those given in \cite[Theorem~7.2]{Savage14}, \cite[Theorem~7.1]{CMS16} and \cite[Theorem~3.9]{CM19}, are up to application of $\Pi$, that is, either $L$ or $\Pi L$ is of the form \eqref{eq:class.irred}.
\end{rem}

Recall that (super)dimensions and (super)characters of evaluation modules are the same as the (super)dimensions and (super)characters of their respective underlying $\fg$-modules. We finish this section by describing (super)dimensions and (super)characters of irreducible Kac-like modules.

\begin{prop} \label{prop:dim+char}
Let $\sm \subseteq A$ be a maximal ideal, $\Theta \in \fz[A]^*$ be such that $\sm^n \subseteq \kteta \subsetneq \sm$, and $\sV$ be a finite-dimensional irreducible $\fg_0'$-module.  Then
\begin{gather*}
\dim \bK_{A/\kteta} (\Theta \boxtimes \sV^\sm)
= 2^{\dim \fg_{-1} \dim A/\kteta} \dim \sV,
\\
\sdim \bK_{A/\kteta} (\Theta \boxtimes \sV^\sm)
= 0,
\\
\ch \bK_{A/\kteta} (\Theta \boxtimes \sV^\sm)
= \ch \bLambda(\fg_{-1})^{\dim(A/\kteta)} \ \ch(\Theta \boxtimes \sV),
\\
\sch \bK_{A/\kteta} (\Theta \boxtimes \sV^\sm) 
= \sch \bLambda(\fg_{-1})^{\dim(A/\kteta)} \ \ch(\Theta \boxtimes \sV).
\end{gather*}
\end{prop}
\begin{proof}
These formulas follow from the fact that $\bK_{A/\kteta}(\Theta \boxtimes \sV^\sm)$ is isomorphic, as a $\fg_0$-module, to $\bLambda(\fg_{-1})^{\otimes \dim(A/\kteta)} \otimes (\Theta \boxtimes \sV)$.
\end{proof}

Notice that, using Theorem~\ref{theo:class.irreps} and Proposition~\ref{prop:dim+char}, one can in fact completely determine the (super)dimension and (super)characters of every finite-dimensional irreducible $\fg[A]$-module, as the (super)character of $\fg_{-1}$ is known (see \cite[Proposition~2.1.2]{Kac77}).
\details{ When $\fg \cong \mathfrak{sl}(m|n)$ and $1 \le m < n$, then $\fg_0 \cong \fsl(m)\oplus \fsl(n)\oplus \C$ and $\fg_{-1} \cong (\C^m)^* \otimes (\C^n) \otimes \C$ as a $\fg_0$-module;  when $\fg \cong \mathfrak{sl}(n|n)$ and $1 < n$, then $\fg_0 \cong \fsl(n)\oplus \fsl(n)$ and $\fg_{-1} \cong (\C^n)^* \otimes (\C^n)$ as a $\fg_0$-module;  and when $\fg \cong \mathfrak{osp}(2|2n)$ and $1 < n$, then $\fg_0 \cong \mathfrak{sp}(2n) \oplus \C$ and $\fg_{-1} \cong \C^{2n} \otimes \C$ as a $\fg_0$-module.}

\section{Extensions}  
\label{sec:exts}

Let $\fa$ be a Lie superalgebra and recall that $\bU(\fa)$ is an associative unital ring. Given two $\fa$-modules $V$ and $W$, the set $\Hom_{\fa}(V,W)$ (or, equivalently, $\Hom_{\bU(\fa)}(V,W)$) consists of all (not just $\Z_2$-graded) homomorphisms of $\fa$-modules (or, equivalently, $\bU(\fa)$-modules). That is,
    \[
    \Hom_{\fa}(V,W)
    := \{ \varphi \in \Hom_\C(V,W) \mid \varphi(xv) = x\varphi(v) \textup{ for all } x \in \fa, v \in V \}.
    \]
Recall that the category of all $\bU(\fa)$-modules has enough projectives, and consider a projective resolution of an object $V$ in this category, $\dotsm \to \sP_1 \to \sP_0 \to V \to 0$. Thus, $\Ext_{\fa}^n(V,W)$, the $n$-th extension group between two $\fa$-modules $V$ and $W$, is the $n$-homology of the cocomplex
    \[
    0
    \to \Hom_{\fa} (P_0,\, W)
    \to \Hom_{\fa} (P_1,\, W)
    \to \Hom_{\fa} (P_2,\, W)
    \to \dotsm
    \]
For instance, using the Koszul complex (see \cite[Section~16.3]{Musson12}), one can see that $\Ext_{\fa}^n(V,W)$ is isomorphic to the $n$-th homology of the cocomplex
\[
0
\to \Hom_\C(V,W)
\to \Hom_\C(\fa, \Hom_{\C}(V,W))
\to \Hom_\C(\Lambda^2\fa, \Hom_{\C}(V,W))
\to \cdots
\]

\details{
We point out that even when $V$ is not $\Z_2$-graded, the resolution $\dotsm \to \sP_1 \to \sP_0 \to V \to 0$ can be taken so that all $P_i$ are $\Z_2$-graded and all maps $P_i\to P_{i-1}$ are even. Moreover, in the case that $V$ is $\Z_2$-graded, the map $P_0\to V$ can be chosen to be even as well (the Koszul resolution \cite[Section~16.3]{Musson12}). Hence, if $V$ and $W$ are $\Z_2$-graded, then the spaces $\Hom_{\fa}(P_i,W)$ and hence the space $\Ext_{\fa}^n(V,W)$ inherits the $\Z_2$-grading from $\Hom_{\C}(P_i,W)$. The even part $\Ext_{\fa}^n(V,W)_{\bar 0}$ is nothing but the $n$-th extension group between $V$ and $W$ in the category of $\Z_2$-graded modules.}

Assume for the remainder of this section that $\fg \cong \fsl(m|n)$ or $\fosp(2|2n)$, $1 \le m \le n$, $1 < n$, that is, that $\fg$ is a basic classical Lie superalgebra of type~I or a finite-dimensional Lie superalgebra isomorphic to $\fsl(n|n)$, and also that $A$ is an associative, commutative, finitely-generated, unital algebra over $\C$.  Our main goal is to describe 1-extensions between finite-dimensional irreducible $\fg[A]$-modules.  By \cite[Theorem~5.7]{CM19}, this boils down to describing 1-extensions between (generalized) evaluation modules supported at the same maximal ideal.  

We begin with a few technical lemmas that will be needed in the proof of Theorem~\ref{thm:ext.kac}.  This first lemma reduces the computation of homomorphisms and extensions between certain Kac-like modules to the computation of homomorphisms and extensions between $\fg_0[A]$-modules.

\begin{lem} \label{lem:hom.ext.kac}
Let $\sM, \sN$ be finite-dimensional $\fg_0[A]$-modules.
\begin{enumerate}
\item \label{lem:hom.ext.kac.a}
If $\sN$ is irreducible, then $\bV_A(\sN)^{\fg_1[A]} = \sN$.

\item \label{lem:hom.ext.kac.b}
If $\sN$ is irreducible, there is a natural isomorphism of vector spaces $\Hom_{\fg[A]} (\bK_A(\sM),\, \bV_A(\sN)) \cong \Hom_{\fg_0[A]} (\sM, \sN)$.

\item \label{lem:hom.ext.kac.c}
If $\sN$ is irreducible and $\Hom_{\fh}(\sM, \Hom_\C(\fg_1[A], \sN)) = 0$, there is an isomorphism of vector spaces $\Ext^1_{\fg_+[A]}(\sM, \bV_A(\sN)) \cong \Ext^1_{\fg_0[A]}(\sM,\sN)$.

\item \label{lem:hom.ext.kac.d}
If $\sM, \sN$ are irreducible $\fg_0[A]$-modules and $\Hom_{\fh}(\sM, \Hom_\C(\fg_1[A], \sN)) = 0$, there is an isomorphism of vector spaces $\Ext^1_{\fg[A]} (\bK_{A}(\sM), \, \bV_A(\sN)) \cong \Ext^1_{\fg_0[A]} (\sM, \sN)$.
\end{enumerate}
\end{lem}
\begin{proof}
\ref{lem:hom.ext.kac.a}: 
First notice that $\sN \subseteq \bV_A(\sN)^{\fg_1[A]}$.  In order to prove that $\bV_A(\sN)^{\fg_1[A]} \subseteq \sN$, fix $v \in \bV_A(\sN)^{\fg_1[A]}$, let $W$ denote the unique maximal $\fg[A]$-submodule of $\bK_A(\sN)$ and $\bLambda^+ (\fg_{-1}[A])$ denote the augmentation ideal of $\bLambda (\fg_{-1}[A])$.  Thus, by the definition of $\bK_A(-)$ \eqref{def:kac-like}, there is an isomorphism of vector spaces $\bK_A(\sN) \cong \left( \bLambda^+ (\fg_{-1}[A]) \otimes \sN \right) \oplus \sN$.  Thus, $W$ identifies as a subspace of $\bLambda^+(\fg_{-1}[A]) \otimes \sN$.  Let $n \in \sN$ and $m \in \bLambda^+(\fg_{-1}[A]) \otimes \sN$ be such that $v = m + n + W$.  We claim that $m+W$ must be zero, thus implying that $v = n \in \sN$.
\details{
Since $v \in \bV_A(\sN)^{\fg_1[A]}$ and $n \in \sN \subseteq \bV_A(\sN)^{\fg_1[A]}$, then $\fg_1[A] (m + W) \subseteq W$.  Moreover, since $\bV_A(\sN)$ is irreducible, if $m+W$ were nonzero, then $m+W$ would generate $\bV_A(\sN)$ as a $\fg[A]$-module; that is, $\bV_A(\sN) = \bU(\fg_{-1}[A])\bU(\fg_0 [A])(m+W)$.  This would be an absurd, since the $\fz$-weights of the elements of $\bU(\fg_{-1}[A])\bU(\fg_0 [A])(m+W)$ are all distinct from the $\fz$-weights of the elements of $\sN$.
}

\ref{lem:hom.ext.kac.b}:
By the definition of $\bK_{A}(-)$~\eqref{def:kac-like} and the tensor-hom adjunction, we have an isomorphism of vector spaces
    \[
    \Hom_{\fg[A]} (\bK_{A}(\sM), \, \bV_A(\sN))
    \cong \Hom_{\fg_+[A]} (\sM, \, \bV_A(\sN)),
    \]
that is natural on both $\sM$ and $\sN$.  Further, since $\sM$ is trivial as a $\fg_1[A]$-module, we have an isomorphism of vector spaces 
    \[
    \Hom_{\fg_+[A]} (\sM, \, \bV_A(\sN))
    \cong \Hom_{\fg_0[A]} (\sM, \, \bV_A(\sN)^{\fg_1[A]}),
    \]
that is natural on both $\sM$ and $\sN$.  Since $\sN$ is assumed to be irreducible, it follows from part~\ref{lem:hom.ext.kac.a} that $\bV_A(\sN)^{\fg_1[A]} = \sN$.  Hence $\Hom_{\fg_0[A]} (\sM, \bV_A(\sN)^{\fg_1[A]}) \cong \Hom_{\fg_0[A]} (\sM, \sN)$.

\ref{lem:hom.ext.kac.c}:
Consider the Lyndon-Hochschild-Serre spectral sequence
    \[
    E_2^{p,q}
    \cong \Ext^p_{\fg_0[A]} (\sM,\, \HH^q (\fg_1[A], \bV_A(\sN)))
    \Rightarrow \Ext^{p+q}_{\fg_+[A]} (\sM,\, \bV_A(\sN)).
    \]
Observe that, in this second page, the only terms contributing to $\Ext^1_{\fg_+[A]} (\sM,\, \bV_A(\sN))$ are 
    \[
    E_2^{1,0}
    \cong \Ext^1_{\fg_0[A]} (\sM,\, \bV_A(\sN)^{\fg_1[A]})
    \cong \Ext^1_{\fg_0[A]} (\sM,\, \sN)
    \quad \textup{and} \quad
    E_2^{0,1}
    \cong 
    \Hom_{\fg_0[A]} (\sM,\, \HH^1(\fg_1[A],\, \bV_A(\sN))).
    \]
We claim that $E_2^{0,1} = 0$.  In fact, using the Koszul complex, one can see that $\HH^1(\fg_1[A], \bV_A(\sN))$ is a subquotient of $\Hom_\C (\fg_1[A], \bV_A(\sN))$.  Since $\Hom_{\fh}(\sM, \Hom_\C (\fg_1[A], \bV_A(\sN))) = 0$ by hypothesis, then $\Hom_{\fh} (\sM,\, \HH^1(\fg_1[A],\, \bV_A(\sN))) = 0$.  As a consequence, $\Hom_{\fg_0[A]} (\sM,\, \HH^1(\fg_1[A],\, \bV_A(\sN))) = 0$.  This implies that
    \[
    \Ext^{1}_{\fg_+[A]} (\sM, \bV_A(\sN))
    \cong E_2^{1,0}
    \cong \Ext^1_{\fg_0[A]} (\sM, \sN).
    \]

\ref{lem:hom.ext.kac.d}:
Consider $\sM$ as a $\fg_+[A]$-module via inflation and then a projective resolution of $\sM$ in the category of $\fg_+[A]$-modules, $\dotsm \to \sP_1 \to \sP_0 \to \sM \to 0$.  Now, notice that the PBW~Theorem implies that $\bU(\fg[A])$ is free as a right $\bU(\fg_+[A])$-module.  Hence, using this fact and the definition of $\bK_A(-)$~\eqref{def:kac-like}, one sees that $\dotsm \to \bK_A(\sP_1) \to \bK_A(\sP_0) \to \bK_A(\sM) \to 0$ is a projective resolution of $\bK_A(\sM)$ in the category of $\fg[A]$-modules. Now, using the definition of $\bK_A(-)$~\eqref{def:kac-like} and the tensor-hom adjunction, one sees that $\Ext^1_{\fg[A]} (\bK_A(\sM), \bV_A(\sN)) \cong \Ext^1_{\fg_+[A]}(\sM, \bV_A(\sN))$.
\details{
recall that $\Ext^1_{\fg[A]} (\bK_A(\sM), \bV_A(\sN))$ is the first cohomology of the cochain complex
    \[
    0
    \to \Hom_{\fg[A]} (\bK_A(\sP_0),\, \bV_A(\sN))
    \to \Hom_{\fg[A]} (\bK_A(\sP_1),\, \bV_A(\sN))
    \to \Hom_{\fg[A]} (\bK_A(\sP_2),\, \bV_A(\sN))
    \to \dotsm
    \]
Hence, using the definition of $\bK_A(-)$~\eqref{def:kac-like} and the tensor-hom adjunction, one sees that this cochain complex is isomorphic to the cochain complex
    \[
    0
    \to \Hom_{\fg_+[A]} (\sP_0,\, \bV_A(\sN))
    \to \Hom_{\fg_+[A]} (\sP_1,\, \bV_A(\sN))
    \to \Hom_{\fg_+[A]} (\sP_2,\, \bV_A(\sN))
    \to \dotsm
    \]
Since the first cohomology of this cochain complex is, by definition, $\Ext^1_{\fg_+[A]}(\sM, \bV_A(\sN))$, it follows that
    \[
    \Ext^1_{\fg[A]} (\bK_A(\sM), \bV_A(\sN))
    \cong \Ext^1_{\fg_+[A]}(\sM, \bV_A(\sN)).
    \]
}
The result follows from part~\ref{lem:hom.ext.kac.c}.
\end{proof}

The next lemma describes the structure of the unique maximal proper $\fg[A]$-submodule of Kac-like modules.  In fact, let $\sL$ be a finite-dimensional irreducible $\fg_0[A]$-module.  Recall that the Kac-like module $\bK_A(\sL)$ has a unique irreducible quotient $\bV_A(\sL)$, and hence a unique maximal proper $\fg[A]$-submodule. In the particular case where $\sL$ is an evaluation module, say $\sL = \Theta \boxtimes \sV^\sm$ for some maximal ideal $\sm \subseteq A$, a finite-dimensional irreducible $\fg_0'$-module $\sV$ and $\Theta \in \fz[A]^*$ such that $\Theta(\fz \otimes \sm) = 0$, then the Kac-like module $\bK_\C(\sL)$ coincides with the usual Kac module $K(\Theta \boxtimes \sV)$ (see \cite[Section~2]{Kac78}) and the irreducible module $\bV_\C(\sL)$ coincides with the irreducible quotient of $K(\Theta \boxtimes \sV)$.  Let $Z_{\Theta \boxtimes \sV}$ denote the unique maximal proper $\fg$-submodule of $K(\Theta \boxtimes \sV)$.

\begin{lem} \label{lem:W=oplus}
Let $\sm \subseteq A$ be a maximal ideal, $\sV_1, \sV_2$ be finite-dimensional irreducible $\fg'_0$-modules, $\Theta_1, \Theta_2 \in \fz[A]^*$ be such that $\Theta_1(\fz \otimes \sm^n) = \Theta_2(\fz \otimes \sm^n) = 0$ for some $n > 0$.  Denote by $W_1$ the maximal proper $\fg[A]$-submodule of $\bK_A(\Theta_1 \boxtimes \sV_1^\sm)$, and consider the following subspace of $\bK_A(\Theta_1 \boxtimes \sV_1^\sm)$:
    \[
    \Omega := \bigoplus_{i \ge 0} \Omega_i,
    \qquad \textup{where} \quad
    \Omega_i := \bLambda^i (\fg_{-1}[A]) \wedge (\fg_{-1} [\kteta_1]) \otimes_\C (\Theta_1 \boxtimes \sV_1).
    \]
\begin{enumerate}
\item \label{lem:W=oplus.a}
If $\sk_{\Theta_1} \ne \sm$, then $W_1$ is isomorphic to $\Omega$ as vector spaces.

\item \label{lem:W=oplus.b}
If $\sk_{\Theta_1} = \sm$, then $W_1$ is isomorphic to $\Omega \oplus Z_{\Theta_1 \boxtimes \sV_1}^\sm$ as vector spaces.

\item \label{lem:W=oplus.c}
If either $\sk_{\Theta_1} \ne \sm$ or $\sk_{\Theta_2} \ne \sm$, then there is an isomorphism of vector spaces
    \[
    \Hom_{\fg[A]} (W_1, \bV_A(\Theta_2 \boxtimes \sV_2^\sm))
    \cong \Hom_{\fg_0[A]}(\fg_{-1}[\sk_{\Theta_1}] \otimes (\Theta_1 \boxtimes \sV_1^\sm),\, \Theta_2 \boxtimes \sV_2^\sm).
    \]
\end{enumerate}
\end{lem}    
\begin{proof}
Set $\sL_1:=\Theta_1 \boxtimes \sV_1^\sm$ and $\sL_2:=\Theta_2 \boxtimes \sV_2^\sm$. Given $u = (x_1 \otimes a_1) \cdots (x_n \otimes a_n) \in \bU(\fg[A])$, let ${\bar u} \in \bU(\fg[A/\sk_{\Theta_1}])$ denote the monomial $(x_1\otimes (a_1+\sk_{\Theta_1})) \cdots (x_n\otimes (a_n+\sk_{\Theta_1}))$. By the definition of $\bK_A(-)$~\eqref{def:kac-like} and the universal property of $\otimes_{\bU(\fg_+[A])}$, there exists a unique $\fg_+[A]$-module map 
    \[
    \psi: \bK_{A} (\sL_1) \to \bK_{A/\sk_{\Theta_1}} (\sL_1)
    \qquad \textup{such that} \qquad
    \psi(u \otimes v) = {\bar u} \otimes v
    \quad \textup{ for all } \ u \in \bU(\fg[A]), \ v \in \sL_1.
    \]
Notice, moreover, that for all $u_1 \in \bU(\fg_{-1}[A])$, $u \in \bU(\fg[A])$, $v \in \sL_1$, we have $\psi(u_1u \otimes v) = \bar{u_1} \bar u \otimes v = u_1 \psi(u \otimes v)$. Hence, $\psi$ is in fact a surjective homomorphism of $\fg[A]$-modules.

\ref{lem:W=oplus.a}:
If $\sk_{\Theta_1}$ is assumed to not be a maximal ideal of $A$, then it follows from Proposition~\ref{prop:main} that $\bK_{A/\sk_{\Theta_1}}(\sL_1)$ is an irreducible $\fg[A]$-module.  Hence the maximal proper submodule of $\bK_A(\sL_1)$ is $\ker\psi$.  In order to show that $\ker\psi = \Omega$, first notice from the definition of $\psi$ that $\Omega_i \subseteq \ker\psi$ for all $i \in \Z_{\ge0}$.  Hence $\Omega \subseteq \ker\psi$.  The containment $\ker\psi \subseteq \Omega$ follows from the PBW~Theorem by choosing a basis for $\sk_{\Theta_1}$ and completing it to a basis for $A$.
\details{
It follows from the PBW~Theorem that $\bK_A(\sL_1) \cong \bU(\fg_{-1}[A])\otimes_\C \sL_1$ and $\bK_{A/\sk_{\Theta_1}} (\sL_1) \cong \bU(\fg_{-1}[A/\sk_{\Theta_1}])\otimes_\C \sL_1$ as vector spaces.  Write $A = (A/\sk_{\Theta_1}) \oplus \kteta_1$ as a vector space, choose a basis $\{b_i\mid i\in \mathtt I\}$ for $A/\sk_{\Theta_1}$, a basis $\{k_j\mid j\in \mathtt J\}$ for $\sk_{\Theta_1}$, a basis $\{ y_1, \dotsc, y_n \}$ for $\fg_{-1}$, and a basis $\{v_1,\ldots, v_l\}$ for $\sL_1$. Then $\{y_p \otimes c_i,\, y_p\otimes k_j \mid p \in \{1, \dotsc, n\},\, i\in \mathtt I,\, j\in \mathtt J\}$ is a basis of $\fg_{-1}[A]$, and the vectors
    \begin{equation}\label{eq:basis.K_A(L)}
    (y_1\otimes b_{i_1})^{\epsilon_1} \cdots (y_n\otimes b_{i_n})^{\epsilon_n} (y_1\otimes k_{j_1})^{\delta_1}\cdots (y_n\otimes k_{j_r})^{\delta_r}\otimes v_m,
    \end{equation}
where $i_1, \dotsc, i_n \in \mathtt I$, $j_1, \dotsc, j_n \in \mathtt J$, $\epsilon_1, \delta_1, \dotsc, \epsilon_n, \delta_n \in \{0,1\}$, $m \in \{1,\dotsc, l\}$ form a basis for $\bK_A (\sL_1)$. Similarly, the vectors $(y_1\otimes (b_{i_1}+\sk_{\Theta_1}))^{\epsilon_1}\cdots (y_n\otimes (b_{i_n}+\sk_{\Theta_1}))^{\epsilon_n} \otimes v_m$, where $i_1, \dotsc, i_n \in \mathtt I$, $\epsilon_1, \dotsc, \epsilon_n \in \{0,1\}$, $m \in \{1,\dotsc, l\}$, form a basis for $\bK_{A/\sk_{\Theta_1}} (\sL_1)$.  Thus, we conclude that a linear combination of vectors of the form \eqref{eq:basis.K_A(L)} lies in $\ker\psi$ if and only all monomials occurring in this combination are in $\Omega$.  This proves that $\ker \psi \subseteq \Omega$, and hence that $\ker \psi = \Omega$.
}

\ref{lem:W=oplus.b}:
If $\sk_{\Theta_1} = \sm$, then $\sL_1$ is an evaluation module and thus, it can be seen as a $\fg_0$-module.  In this case, the unique (up to scalar) surjective homomorphism of $\fg[A]$-modules $\varphi : \bK_A (\sL_1)\to \bV_\C (\sL_1)^\sm$ factors through $\psi$:
    \[
    \bK_A (\sL_1) \xrightarrow{\psi}
    \bK_{A/\sm} (\sL_1) \cong
    \bK_\C (\sL_1)^\sm \xrightarrow{\pi}
    \bV_\C (\sL_1)^\sm.
    \]
Thus, $\ker \varphi = \psi^{-1}(\ker \pi) = \psi^{-1}(Z_{\Theta_1 \boxtimes \sV_1}^\sm)$. Now, notice that $\bK_A (\sL_1) \cong \Omega \oplus \bK_\C (\sL_1)$ as vector spaces.  Hence, it follows from the definition of $\psi$ that $\psi^{-1}(Z_{\Theta_1 \boxtimes \sV_1}^\sm) = \Omega \oplus Z_{\Theta_1 \boxtimes \sV_1}^\sm$.  This implies that $\ker\varphi$, the maximal proper submodule of $\bK_A(\sL_1)$, is isomorphic to $\Omega \oplus Z_{\Theta_1 \boxtimes \sV_1}^\sm$.

\ref{lem:W=oplus.c}: 
Throughout the proof of this part, we will identify $\Omega \oplus (Z_{\Theta_1 \boxtimes \sV_1}^\sm)^{\delta_{\sm, \sk_{\Theta_1}}}$ with $W_1$, without further reference to the isomorphisms proved in parts \ref{lem:W=oplus.a} and \ref{lem:W=oplus.b}.  

We begin by showing that, if $\sk_{\Theta_2} \ne \sm$, then $\Hom_{\fg[A]} (Z_{\Theta_1 \boxtimes \sV_1}^\sm, \bV_A(\sL_2)) = 0$.  First, notice that there is an isomorphism of vector spaces
    \[
    \Hom_{\fg[A]} (Z_{\Theta_1 \boxtimes \sV_1}^\sm, \bV_A(\sL_2))
    \cong \Hom_{\fg} (Z_{\Theta_1 \boxtimes \sV_1}, \bV_A(\sL_2)^{\fg[\sm]}).
    \]
Then, notice that $\bV_A(\sL_2)^{\fg[\sm]}$ is a $\fg[A]$-submodule of $\bV_A(\sL_2)$.  Since $\sL_2 \subseteq \bV(\sL_2)$ and $\sk_{\Theta_2} \ne \sm$, then $\fg[\sm] \bV_A(\sL_2) \ne 0$, and hence $\bV_A(\sL_2)^{\fg[\sm]} \ne \bV_A(\sL_2)$.  Further, since $\bV_A(\sL_2)$ is irreducible by definition, then $\bV_A(\sL_2)^{\fg[\sm]} = 0$.  This implies that $\Hom_{\fg[A]} (Z_{\Theta_1 \boxtimes \sV_1}^\sm, \bV_A(\sL_2)) = 0$.

We resume the proof by showing that $\Hom_{\fg[A]}(\Omega, \bV_A(\sL_2)) \cong \Hom_{\fg_0[A]} (\fg_{-1}[\sk_{\Theta_1}] \otimes \sL_1, \sL_2)$.  First notice that, since $\Omega_0$ is a $\fg_+[A]$-submodule of $\bK_A(\sL_1)$ such that $\fg_1[A]\Omega_0=0$, then $\bK_A(\Omega_0)$ is well-defined.  Further, via the tensor-hom adjunction, the inclusion map $\Omega_0 \to \Omega$, induces a surjective homomorphism of $\fg[A]$-modules $\xi : \bK_A(\Omega_0) \to \Omega$, which satisfies $\xi(u \otimes \omega) = u\omega$ for all $u \in \bU(\fg[A])$, $\omega \in \Omega_0$. Thus, for every $f \in \Hom_{\fg[A]}(\Omega, \bV_A(\sL_2))$, we can define $(f\circ\xi)\in \Hom_{\fg[A]}(\bK_A (\Omega_0), \bV_A(\sL_2))$.  We claim that the map $\xi^* : \Hom_{\fg[A]}(\Omega, \bV_A(\sL_2)) \to \Hom_{\fg[A]}(\bK_A (\Omega_0), \bV_A(\sL_2))$, $\xi^*(f) = (f \circ \xi)$, is an isomorphism of vector spaces.  Since
    \[
    \Hom_{\fg[A]}(\bK_A (\Omega_0), \bV_A(\sL_2))
    \cong \Hom_{\fg_0[A]}(\Omega_0, \sL_2)
    = \Hom_{\fg_0[A]}(\fg_{-1}[\sk_{\Theta_1}] \otimes \sL_1, \sL_2)
    \]
by Lemma~\ref{lem:hom.ext.kac}\ref{lem:hom.ext.kac.b}, this finishes the proof.

To check that $\xi^*$ is a linear map is straight-forward, and the fact that $\xi^*$ is injective follows directly from the fact that $\xi$ is surjective.  In order to prove that $\xi^*$ is surjective, first notice that a map $h \in \Hom_{\fg[A]} (\bK_A (\Omega_0), \bV_A(\sL_2))$ factors through $\xi$ if and only if $\ker\xi \subseteq \ker h$. The isomorphisms of vector spaces 
    \[
    \bK_A(\Omega_0) 
    \cong \bLambda(\fg_{-1}[A])\otimes_\C \fg_{-1} [\sk_{\Theta_1}] \otimes_\C \sL_1
    \quad \textup{and} \quad
    \Omega
    \cong \bLambda(\fg_{-1}[A])\wedge \fg_{-1} [\sk_{\Theta_1}] \otimes_\C \sL_1
    \]
imply that $\ker\xi$ is the $\fg[A]$-submodule of $\bK_A(\Omega_0)$ generated by vectors $(y \otimes k) \otimes ((y \otimes k) \otimes v)$ such that $y \in \fg_{-1}$, $k\in \sk_{\Theta_1}$ and $v \in \sL_1$.  Thus, let $h \in \Hom_{\fg[A]}(\bK_A (\Omega_0), \bV_A(\sL_2))$ and $w_{y,k,v} := h((y \otimes k) \otimes ((y \otimes k) \otimes v))$.  In order to prove that $h$ factors through $\xi$, we will prove that $w_{y,k,v} = 0$ for all $y \in \fg_{-1}$, $k \in \sk_{\Theta_1}$, $v \in \sL_1$.  First notice that the super anticommutativity of $[\cdot, \cdot]$ and Jacobi identity imply that $[[x,y], y]=0$ for all $x\in \fg_1$, $y\in \fg_{-1}$. Thus, for all $x\in \fg_1$ and $a\in A$ we have that
    \[
    (x\otimes a)w_{y,k,v}
    = (x\otimes a)h((y \otimes k) \otimes ((y \otimes k) \otimes v))
    = h(([x,y], y] \otimes ak^2) \otimes v)=0.
    \]
Since $x \in \fg_1$ and $y \in \fg_{-1}$, then $[x,y] \in \fg_0$ and the super anticommutativity implies that $[[x,y], y] = - [y, [x,y]]$.  Further, by the super Jacobi identity and the fact that $[\fg_{-1}, \fg_{-1}] = 0$, we have that $-[y, [x,y]] = -[[y,x],y]$.  Then using the super anticommutativity again, we obtain that $-[[y,x],y] = -[[x,y],y]$.  This shows that $[[x,y], y] = - [[x,y], y]$.  This implies that $W_{y,k,v} := \bU(\fg[A])w_{y,k,v} = \bLambda(\fg_{-1}[A])\bU(\fg_0[A])w_{y,k,v}$ is a $\fg[A]$-submodule of $\bV_2(\sL_2)$.  Moreover, since $\bV_2(\sL_2)$ is an irreducible $\fg[A]$-module, then $W_{y,k,v}$ cannot be a proper submodule of $\bV_A(\sL_2)$.  Further, since $w_{y,k,v} = (y \otimes k) h((y \otimes k) \otimes v)\in \bLambda^+(\fg_{-1}[A])\sL_2$ , then $W_{y,k,v} \cap \sL_2 = \{0\}$.  This implies that $W_{y,k,v} = 0$, and hence that $w_{y,k,v} = 0$.  The result follows.
\end{proof}

We are now in a position to prove the second main result of the paper, which will be used to describe the block decomposition of the categories of finite-dimensional modules over $\fg[A]$ and $\hat\fg$ in Section~\ref{subsec:blk.dec}.  Compare it to \cite[Theorem~3.7]{NS15} and \cite[Theorem~5.7]{CM19}.

\begin{theo} \label{thm:ext.kac}
Let $\sV_1, \sV_2$ be finite-dimensional irreducible $\fg'_0$-modules, let $\sm$ be a maximal ideal of $A$, let $\Theta_1, \Theta_2 \in \fz[A]^*$, assume that $\Theta_1(\fz \otimes \sm^n) = \Theta_2(\fz \otimes \sm^n) = 0$ for some $n \in \Z_{>0}$ and that either $\Theta_1(\fz \otimes \sm) \ne 0$ or $\Theta_2(\fz \otimes \sm) \ne 0$.
\begin{enumerate}

\item \label{thm:ext.kac.b}
If $\Theta_1 \ne \Theta_2$ and $\Theta_1(z) - \Theta_2(z) \in \Z_{\ge0}$, then
    \[
    \Ext^1_{\fg[A]} \left( \bV_A(\Theta_1 \boxtimes \sV_1^\sm),\, \bV_A(\Theta_2 \boxtimes \sV_2^\sm) \right)
    \cong \Hom_{\fg_0[A]} (\fg_{-1}[\sk_{\Theta_1}] \otimes (\Theta_1 \boxtimes \sV_1^\sm),\, \Theta_2 \boxtimes \sV_2^\sm).
    \]

\item \label{thm:ext.kac.c}
If $\Theta_1 \ne \Theta_2$ and $\Theta_2(z) - \Theta_1(z) \in \Z_{\ge0}$, then
    \begin{align*}
    \Ext^1_{\fg[A]} \left( \bV_A(\Theta_1 \boxtimes \sV_1^\sm),\, \bV_A(\Theta_2 \boxtimes \sV_2^\sm) \right)
    \cong {}&{} \Hom_{\fg_0[A]} (\fg_{-1}[\sk_{\Theta_2}] \otimes (\Theta_2 \boxtimes \sV_2^\sm),\, \Theta_1 \boxtimes \sV_1^\sm).
    \end{align*}
    
\item \label{thm:ext.kac.a}
If $\Theta_1(z) - \Theta_2(z) \notin \Z$, then $\Ext^1_{\fg[A]} \left( \bV_A(\Theta_1 \boxtimes \sV_1^\sm), \bV_A(\Theta_2 \boxtimes \sV_2^\sm) \right) = 0$.

\item \label{thm:ext.kac.d}
In the case where $\Theta_1 = \Theta_2$, we have $\Ext^1_{\fg[A]} \left( \bV_A(\Theta_1 \boxtimes \sV_1^\sm), \, \bV_A(\Theta_2 \boxtimes \sV_2^\sm) \right) = 0$ if and only if $\Ext^1_{\fg_0[A]} \left( \Theta_1 \boxtimes \sV_1^\sm, \, \Theta_2 \boxtimes \sV_2^\sm \right) = 0$.
\end{enumerate}
\end{theo}
\begin{proof}
Throughout the proof we will denote $\Theta_1 \boxtimes \sV_1^\sm$ by $\sL_1$ and $\Theta_2 \boxtimes \sV_2^\sm$ by $\sL_2$. 

\ref{thm:ext.kac.b}: 
Consider the short exact sequence of $\fg[A]$-modules
    \[
    0 \to W_1 \to \bK_A (\sL_1) \to \bV_A (\sL_1) \to 0,
    \]
where $W_1$ is the unique maximal proper submodule of $\bK_A (\sL_1)$.  Applying $\Hom_{\fg[A]} (-, \bV_A(\sL_2))$ to this short exact sequence, we obtain a long exact sequence on $\Ext$.  We can extract from this long exact sequence the following exact sequence:
    \begin{multline*}
    \Hom_{\fg[A]} (\bK_{A} (\sL_1),\, \bV_A(\sL_2))
    \to \Hom_{\fg[A]} (W_1,\, \bV_A(\sL_2))\to \\
    \to \Ext_{\fg[A]}^1 (\bV_A(\sL_1),\, \bV_A(\sL_2)) 
    \to \Ext_{\fg[A]}^1 (\bK_{A} (\sL_1),\, \bV_A(\sL_2)).
    \end{multline*}
Now, notice that Lemma~\ref{lem:hom.ext.kac}\ref{lem:hom.ext.kac.b} implies that $\Hom_{\fg[A]} (\bK_A(\sL_1), \bV_A(\sL_2)) \cong \Hom_{\fg_0[A]} (\sL_1, \sL_2)$, and the fact that $\Theta_1 \ne \Theta_2$ implies that $\Hom_{\fg_0[A]} (\sL_1, \sL_2) = 0$.  Further, the fact that $\Theta_1(z) - \Theta_2(z) \in \Z_{\ge0}$ and Lemma~\ref{lem:hom.ext.kac}\ref{lem:hom.ext.kac.d} imply that $\Ext_{\fg[A]}^1 (\bK_A (\sL_1), \bV_A(\sL_2)) \cong \Ext_{\fg_0[A]}^1 (\sL_1, \sL_2)$, and \cite[Corollary~2.5]{NS15} and the fact that $\Theta_1 \ne \Theta_2$ imply that $\Ext_{\fg_0[A]}^1 (\sL_1, \sL_2) = 0$.  Hence, we have the following isomorphism of vector spaces:
    \[
    \Ext_{\fg[A]}^1 (\bV_A(\sL_1), \bV_A(\sL_2))
    \cong \Hom_{\fg[A]} (W_1, \bV_A(\sL_2)).
    \]
Since we are assuming that either $\sk_{\Theta_1} \ne \sm$ or $\sk_{\Theta_2} \ne \sm$, then Lemma~\ref{lem:W=oplus}\ref{lem:W=oplus.c} yields an isomorphism of vector spaces
    \[
    \Hom_{\fg[A]} (W_1, \bV_A(\sL_2))
    \cong \Hom_{\fg_0[A]} (\fg_{-1}[\sk_{\Theta_1}] \otimes \sL_1, \sL_2).
    \]

\ref{thm:ext.kac.c}: 
First recall from Lemma~\ref{lem:dual.ext} that $\Ext^1_{\fg[A]} (\bV_A(\sL_1), \bV_A(\sL_2)) \cong \Ext^1_{\fg[A]} (\bV_A(\sL_2), \bV_A(\sL_1))$.  Thus, the result of part~\ref{thm:ext.kac.c} follows from that of part~\ref{thm:ext.kac.b} by exchanging the indices $_1$ and $_2$.

The proof of part~\ref{thm:ext.kac.a} is similar to that of part~\ref{thm:ext.kac.b}.
\details{
Denote $\Theta_1 \boxtimes V_1^\sm$ by $L_1$ and $\Theta_2 \boxtimes V_2^\sm$ by $L_2$.  Consider the short exact sequence of $\fg[A]$-modules
    \[
    0 \to W_1 \to \bK_A (L_1) \to \bV_A (L_1) \to 0,
    \]
where $W_1$ is the unique maximal proper submodule of $\bK_A (L_1)$.  Applying $\Hom_{\fg[A]} (-, \bV_A(L_2))$ to this short exact sequence, we obtain a long exact sequence on $\Ext$.  We can extract from this long exact sequence the following exact sequence:
    \begin{multline*}
    \Hom_{\fg[A]} (\bK_{A} (L_1),\, \bV_A(L_2))
    \to \Hom_{\fg[A]} (W_1,\, \bV_A(L_2))\to \\
    \to \Ext_{\fg[A]}^1 (\bV_A(L_1),\, \bV_A(L_2)) 
    \to \Ext_{\fg[A]}^1 (\bK_{A} (L_1),\, \bV_A(L_2)).
    \end{multline*}
Now, notice that Lemma~\ref{lem:hom.ext.kac}\ref{lem:hom.ext.kac.b} implies that $\Hom_{\fg[A]} (\bK_A(L_1), \bV_A(L_2)) \cong \Hom_{\fg_0[A]} (L_1, L_2)$, and the fact that $\Theta_1 \ne \Theta_2$ implies that $\Hom_{\fg_0[A]} (L_1, L_2) = 0$.  Further, the fact that $\Theta_1(z) - \Theta_2(z) \notin \Z$ and Lemma~\ref{lem:hom.ext.kac}\ref{lem:hom.ext.kac.d} imply that $\Ext_{\fg[A]}^1 (\bK_A (L_1), \bV_A(L_2)) \cong \Ext_{\fg_0[A]}^1 (L_1, L_2)$, and \cite[Corollary~2.5]{NS15} and the fact that $\Theta_1 \ne \Theta_2$ imply that $\Ext_{\fg_0[A]}^1 (L_1, L_2) = 0$.  Hence, we have the following isomorphism of vector spaces:
    \[
    \Ext_{\fg[A]}^1 (\bV_A(L_1), \bV_A(L_2))
    \cong \Hom_{\fg[A]} (W_1, \bV_A(L_2)).
    \]
Since we are assuming that either $\sk_{\Theta_1} \ne \sm$ or $\sk_{\Theta_2} \ne \sm$, then by Lemma~\ref{lem:W=oplus}, we have an isomorphism of vector spaces
    \[
    \Hom_{\fg[A]} (W_1, \bV_A(L_2))
    \cong \Hom_{\fg_0[A]} (\fg_{-1}[\sk_{\Theta_1}] \otimes L_1, L_2).
    \]
Since $\Theta_1(z) - \Theta_2(z) \notin \Z$, then $\Hom_{\fh} (\fg_{-1}[\sk_{\Theta_1}] \otimes L_1, L_2) = 0$.  The result of part~\ref{thm:ext.kac.a} follows.
}

\ref{thm:ext.kac.d}: 
The proof of the \textit{if} part is similar to that of part~\ref{thm:ext.kac.b}.
\details{
Assume that $\Ext_{\fg_0[A]}^1 (\sL_1, \sL_2) = 0$, and consider the short exact sequence of $\fg[A]$-modules
    \[
0 \to W_1 \to \bK_A (\sL_1) \to \bV_A (\sL_1) \to 0,
    \]
where $W_1$ is the unique maximal proper submodule of $\bK_A (\sL_1)$.  Applying $\Hom_{\fg[A]} (-, \bV_A(\sL_2))$ to this short exact sequence, we obtain a long exact sequence on $\Ext$.  We can extract from this long exact sequence the following exact sequence:
    \[
\Hom_{\fg[A]} (W_1,\, \bV_A(\sL_2))
\to \Ext_{\fg[A]}^1 (\bV_A(\sL_1),\, \bV_A(\sL_2)) 
\to \Ext_{\fg[A]}^1 (\bK_{A} (\sL_1),\, \bV_A(\sL_2)).
    \]
Using the fact that $\Theta_1 = \Theta_2$, we obtain, from Lemma~\ref{lem:W=oplus}\ref{lem:W=oplus.c}, that $\Hom_{\fg[A]} (W_1, \bV_A(\sL_2))=0$, and from Lemma~\ref{lem:hom.ext.kac}\ref{lem:hom.ext.kac.d}, that $\Ext_{\fg[A]}^1 (\bK_{A} (\sL_1),\, \bV_A(\sL_2))\cong \Ext_{\fg_0[A]}^1 (\sL_1,\, \sL_2)$. Further, since $\Ext_{\fg_0[A]}^1 (\sL_1, \sL_2)$ is assumed to be $0$, then we conclude that $\Ext_{\fg[A]}^1 (\bV_A(\sL_1),\, \bV_A(\sL_2)) = 0$.
}
To prove the \textit{only if} part, assume that $\Ext_{\fg[A]}^1 (\bV_A(\sL_1), \bV_A(\sL_2)) = 0$, and consider the Lyndon-Hochschild-Serre spectral sequence 
\[
E_2^{p,q}
\cong \Ext^p_{\fg[A/\sk_{\Theta_2}]} (\bV_A(\sL_1), \, \HH^q(\fg[\sk_{\Theta_2}], \bV_A(\sL_2)))
\Rightarrow \Ext^{p+q}_{\fg[A]} (\bV_A(\sL_1), \bV_A(\sL_2)).
\]
From its 5-term exact sequence, we can extract the following exact sequence:
    \[
    0 \to 
    \Ext_{\fg[A/\sk_{\Theta_2}]}^1 (\bV_A(\sL_1),\, \bV_A(\sL_2)) \to
    \Ext_{\fg[A]}^1 (\bV_A(\sL_1),\, \bV_A(\sL_2)).
    \]
Since $\Ext_{\fg[A]}^1 (\bV_{A} (\sL_1), \bV_A(\sL_2))$ is assumed to be $0$, it follows that $\Ext_{\fg[A/\sk_{\Theta_2}]}^1 (\bV_{A} (\sL_1), \bV_A(\sL_2))=0$. Further, using the fact that $\bV_{A} (\sL_1)\cong \bK_{A/\sk_{\Theta_1}} (\sL_1)$, the hypothesis that $\Theta_1 = \Theta_2$, and Lemma~\ref{lem:hom.ext.kac}\ref{lem:hom.ext.kac.d} (with $A$ replaced by $A/\kteta$) we get that
    \[
    0
    = \Ext_{\fg[A/\sk_{\Theta_2}]}^1 (\bV_{A} (\sL_1),\, \bV_A(\sL_2))
    \cong \Ext_{\fg[A/\sk_{\Theta_1}]}^1 (\bK_{A/\sk_{\Theta_1}} (\sL_1),\, \bV_A(\sL_2))
    \cong \Ext_{\fg_0[A/\sk_{\Theta_1}]}^1 (\sL_1,\, \sL_2).
    \]
Now, using the fact that $\fg_0[A/\sk_{\Theta_1}] = \fz[A/\sk_{\Theta_1}] \oplus \fg'_0[A/\sk_{\Theta_1}]$ as Lie algebras, the K\"unneth formula, the fact that $\Theta_1 = \Theta_2$, \cite[Corollary~2.5 and Theorem~3.7]{NS15}, we obtain the following isomorphisms of vector spaces:
\begin{align*}
\Ext_{\fg_0[A/\sk_{\Theta_1}]}^1 (\sL_1,\, \sL_2)
\cong& \Hom_{\fz[A/\sk_{\Theta_1}]}(\Theta_1, \Theta_2) \otimes \Ext_{\fg_0'[A/\sk_{\Theta_1}]}^1 (\sV_1^\sm,\, \sV_2^\sm) \nonumber \\
    & \oplus \Ext^1_{\fz[A/\sk_{\Theta_1}]}(\Theta_1, \Theta_2) \otimes \Hom_{\fg_0'[A/\sk_{\Theta_1}]}(\sV_1^\sm, \sV_2^\sm) \nonumber \\
\cong&{\ } \left( \C \otimes \Hom_{\fg_0'} (\fg_0' \otimes \sV_1^\sm, \sV_2^\sm)^{\dim_{A/\sm} \sm/\sm^2} \right) \oplus \left( \fz[A/\sk_{\Theta_1}]^* \otimes \Hom_{\fg_0'}(\sV_1, \sV_2) \right).
\end{align*}
Since $\Ext_{\fg_0[A/\sk_{\Theta_1}]}^1 (\sL_1,\, \sL_2) = 0$, it follows that $\Hom_{\fg_0'} (\fg_0' \otimes \sV_1^\sm, \sV_2^\sm) = \Hom_{\fg_0'} (\sV_1, \sV_2) = 0$.  Hence, it follows from the above isomorphism (with $A/\sk_{\Theta_1}$ replaced by $A$) that $\Ext_{\fg_0[A]}^1 (\sL_1,\, \sL_2) = 0$.
\end{proof}

\begin{rem}
Let $\sm \subseteq A$ be a maximal ideal, $\sV_1, \sV_2$ be finite-dimensional irreducible $\fg'_0$-modules, and $\Theta_1, \Theta_2 \in \fz[A]^*$ be such that $\sk_{\Theta_1} = \sk_{\Theta_2} = \sm$.  It has been proven in \cite[Theorem~5.7(b)]{CM19} that there is an isomorphism of vector spaces
    \begin{align*}
    \Ext^1_{\fg[A]} \left( \bV_A(\Theta_1 \boxtimes \sV_1^\sm),\, \bV_A(\Theta_2 \boxtimes \sV_2^\sm) \right)
    \cong {}&{} \Ext^1_{\fg} (\bV_\C (\Theta_1 \boxtimes \sV_1^\sm),\, \bV_\C (\Theta_2 \boxtimes \sV_2^\sm)) \\
        {}&{}\oplus \Hom_{\fg} (\fg \otimes \bV_\C (\Theta_1 \boxtimes \sV_1^\sm),\, \bV_\C (\Theta_2 \boxtimes \sV_2^\sm))^{\dim_{A/\sm} (\sm/\sm^2)}.
    \end{align*}
\end{rem}

The following result implies that $\fg[A]$ is \emph{extension-local}, as defined in \cite[Definition~5.13]{NS15}.

\begin{cor} \label{cor:gA.ext-loc}
If $V$ is a finite-dimensional irreducible evaluation $\fg[A]$-module and $\C_\lambda$ is a 1-dimen\-sion\-al irreducible generalized evaluation $\fg[A]$-module, then $\Ext^1_{\fg[A]}(V, \C_\lambda) = 0$.
\end{cor}
\begin{proof}
First let $V$ be an evaluation module evaluated at a maximal ideal $\sm_1 \subseteq A$ and $\C_\lambda$ be a generalized evaluation module evaluated at a maximal ideal $\sm_2 \subseteq A$.  By \cite[Theorem~5.7(a)]{CM19}, if $\sm_1 \ne \sm_2$, then $\Ext^1_{\fg[A]} (V, \C_\lambda) = 0$.  Now, assume that $\sm_1 = \sm_2$, denote $\sm := \sm_1 = \sm_2$, let $\Theta_1 \in \fh[A]^*$ and $\sV_1$ be a finite-dimensional irreducible $\fg$-module such that $V = \bV_A(\Theta_1 \boxtimes \sV^\sm)$, and let $\Theta_2 \in \fh[A]^*$ be such that $\C_\lambda = \bV_A(\Theta_2 \boxtimes \C^\sm)$.

If $\Theta_1(z)-\Theta_2(z) \notin \Z$, then $\Ext^1_{\fg[A]}(V, \C_\lambda) = 0$ by Theorem~\ref{thm:ext.kac}\ref{thm:ext.kac.a}.  Else, if $\Theta_1(z)-\Theta_2(z) \in \Z_{\ge0}$, then, by Theorem~\ref{thm:ext.kac}\ref{thm:ext.kac.b}, we have:
\[
\Ext^1_{\fg[A]} (V, \C_\lambda)
\cong \Hom_{\fg_0[A]} (\fg_{-1}[\sk_{\Theta_2}] \otimes (\Theta_2 \boxtimes \C^\sm), \Theta_1 \boxtimes \sV_1^\sm).
\]
Further, using the tensor-hom adjunction and the fact that $\Theta_2^* = - \Theta_2$, we see that
\[
\Hom_{\fg_0[A]} (\fg_{-1}[\sk_{\Theta_2}] \otimes (\Theta_2 \boxtimes \C^\sm), \Theta_1 \boxtimes \sV_1^\sm) \cong \Hom_{\fg_0[A]} (\fg_{-1}[\sk_{\Theta_2}], (\Theta_1-\Theta_2) \boxtimes \sV_1^\sm).
\]
To finish the proof of this case, we will show that, if $\Hom_{\fg_0[A]} (\fg_{-1}[\sk_{\Theta_2}], (\Theta_1-\Theta_2) \boxtimes \sV_1^\sm) \ne 0$, then $(\Theta_1-\Theta_2) \boxtimes \sV_1^\sm$ contains a copy of the evaluation $\fg_0[A]$-module $\fg_{-1}^\sm$.  In fact, consider a filtration of $\fg_{-1}[\sk_{\Theta_2}]$ induced by powers of $\sm$: $\fg_{-1}[\sk_{\Theta_2}] \supsetneq \fg_{-1}[\sk_{\Theta_2} \cap \sm] \supsetneq \fg_{-1}[\sk_{\Theta_2} \cap \sm^2] \supsetneq \dotsb$.  Since $\fg_{-1}[\sk_{\Theta_2} \cap \sm^i] / \fg_{-1}[\sk_{\Theta_2} \cap \sm^{i+1}]$ is a direct sum of copies of $\fg_{-1}^\sm$, then every irreducible factor of $\fg_{-1}[\sk_{\Theta_2}]$ must be isomorphic to $\fg_{-1}^\sm$, and hence, the socle of $\fg_{-1}[\sk_{\Theta_2}]$ must be isomorphic to $(\fg_{-1}^\sm)^{\oplus p}$ for some $p > 0$.  Thus, if there exists a non-zero $f \in \Hom_{\fg_0[A]} (\fg_{-1}[\sk_{\Theta_2}], (\Theta_1-\Theta_2) \boxtimes \sV_1^\sm)$, then the image of the socle of $\fg_{-1}[\sk_{\Theta_2}]$ by $f$ must contain at least one copy of $\fg_{-1}^\sm$.  Since $\fg_{-1}^\sm$ is an evaluation module and $(\Theta_1-\Theta_2) \boxtimes \sV_1^\sm$ is an irreducible generalized evaluation module, then $(\Theta_1-\Theta_2) \boxtimes \sV_1^\sm$ cannot contain any copy of $\fg_{-1}^\sm$.  This ultimately implies that $\Ext^1_{\fg[A]} (V, \C_\lambda) = 0$.  

The case where $\Theta_2(z)-\Theta_1(z) \in \Z_{\ge0}$ follows from Theorem~\ref{thm:ext.kac}\ref{thm:ext.kac.c} and an argument very similar to the one used in the paragraph above.
\end{proof}

\begin{rem}
In the case where $A = \C[t]$, $\sm = (t)$, $\Theta_1 \boxtimes \sV_1^\sm = \C$ (trivial $\fg_0[A]$-module), $\Theta_2(z) = \Theta_2(z \otimes t) \ne 0$, $\Theta_2(z \otimes t^n) = 0$ for all $n \ge 2$, and $\sV_2 = \C$ (trivial $\fg_0'$-module), then Theorem~\ref{thm:ext.kac}\ref{thm:ext.kac.b}-\ref{thm:ext.kac.a} imply that $\Ext^1_{\fg[A]} (\C, \bV_A(\Theta_2 \boxtimes \C^\sm)) = 0$.
\details{
If $\Theta_2(z) \notin \Z$, then it follows directly from Theorem~\ref{thm:ext.kac}\ref{thm:ext.kac.a} that $\Ext^1_{\fg[A]} (\C, \bV_A(\Theta_2 \boxtimes \C^\sm)) = 0$.  If $\Theta_2(z) \in \Z_{<0}$, then it follows from Theorem~\ref{thm:ext.kac}\ref{thm:ext.kac.b} that
    \[
    \Ext^1_{\fg[A]} (\C, \bV_A(\Theta_2 \boxtimes \C^\sm))
    \cong \Hom_{\fg_0[A]} (\fg_{-1}[\sm], \Theta_2 \boxtimes \C^\sm).
    \]
Since, as a $\fg_0'$-module, $\fg_{-1}$ is irreducible and not isomorphic to $\C$, then $\Hom_{\fg_0'} (\fg_{-1}[\sm], \Theta_2 \boxtimes \C^\sm) = 0$.  Thus, $\Hom_{\fg_0[A]} (\fg_{-1}[\sm], \Theta_2 \boxtimes \C^\sm) = 0$.  If $\Theta_2(z) \in \Z_{>0}$, then it follows from Theorem~\ref{thm:ext.kac}\ref{thm:ext.kac.c} that
    \[
    \Ext^1_{\fg[A]} (\C, \bV_A(\Theta_2 \boxtimes \C^\sm))
    \cong \Hom_{\fg_0[A]} (\fg_{-1}[\sk_{\Theta_2}] \otimes (\Theta_2 \boxtimes \C^\sm), \C).
    \]
Since $\Hom_{\fg_0[A]} (\fg_{-1}[\sk_{\Theta_2}] \otimes (\Theta_2 \boxtimes \C^\sm), \C) \cong \Hom_{\fg_0[A]} (\fg_{-1}[\sk_{\Theta_2}], (\Theta_2 \boxtimes \C^\sm)^*)$, $(\Theta_2 \boxtimes \C^\sm)^* \cong \Theta_2^* \boxtimes \C^\sm$, and as a $\fg_0'$-module, $\fg_{-1}$ is irreducible and not isomorphic to $\C$, then $\Hom_{\fg_0'} (\fg_{-1}[\sk_{\Theta_2}], \Theta_2^* \boxtimes \C^\sm) = 0$.  The claim follows.
}
This shows that some of the calculations regarding the kernel of the transgression map \cite[Example~5.9]{CM19} are incorrect.
\end{rem}

\section{Applications}  
\label{sec:apps}

In this section, assume that $\fg$ is a classical Lie superalgebra and $A$ is an associative, commutative, finitely-generated, unital algebra, both defined over $\C$.

\subsection{Blocks} \label{subsec:blk.dec}

In this subsection, we will describe the block decomposition of the categories of finite-dimensional modules for classical map and affine superalgebras.

Denote by $\mathcal F$ the category of finite-dimensional $\fg[A]$-modules.  Then, consider the smallest equivalence relation on the set of finite-dimensional irreducible $\fg[A]$-modules such that $V_1$ is equivalent to $V_2$ if, either $V_1$ is isomorphic to $V_2$, or $\Ext^1_{\fg[A]}(V_1, V_2) \ne 0$. The equivalence class of an irreducible module $V\in \mathcal F$ under this equivalence relation will be denoted by $[V]$, and will be called the block of $\mathcal F$ corresponding to $[V]$. Let $\mathtt B$ be the set of all blocks of $\mathcal F$.

Recall that every object in $\mathcal F$ has a Jordan-Hölder series.  For each $\mathtt b \in \mathtt B$, denote by $\mathcal F_{\mathtt b}$ the full subcategory of $\mathcal F$ whose objects are the finite-dimensional $\fg[A]$-modules whose composition factors lie entirely in $\mathtt b$.  Further, for each $M$ in $\mathcal F$ and $\mathtt b \in \mathtt B$, denote by $M_{\mathtt b}$ the sum of submodules of $M$ which lie in $\mathcal F_{\mathtt b}$.  One can show that $M = \bigoplus_{\mathtt b \in \mathtt B} M_{\mathtt b}$ and $\Hom_{\fg[A]}(M, M') \cong \bigoplus_{\mathtt b} \Hom_{\mathcal F_{\mathtt b}} (M_{\mathtt b}, M'_{\mathtt b})$ (cf. \cite[Lemma~5.5]{NS15}).  Hence, $\mathcal F = \bigoplus_{\mathtt b \in \mathtt B} \mathcal F_{\mathtt b}$ is said to be the \emph{block decomposition} of $\mathcal F$ (cf. \cite[Corollary~5.6]{NS15}).

Notice that, to describe the blocks $\mathcal F_{\mathtt b}$, we only need to determine which irreducible modules lie in which block.  In order to describe a given block, for each $\sm \in \MaxSpec(A)$, denote by $\mathcal F_{\sm}$ the full subcategory of $\mathcal F$ whose objects are finite-dimensional $\fg[A]$-modules whose composition factors are (generalized) evaluation modules evaluated at $\sm$.  As in the paragraph above, denote the blocks of the category $\mathcal F_\sm$ by $\mathtt B_{\sm}$.  Then, for each finite-dimensional irreducible $\fg[A]$-module $V = V_1^{\sm_1^{p_1}} \otimes \dotsm \otimes V_k^{\sm_k^{p_k}}$, define its \emph{spectral character} to be the function
\[
\chi_V : \MaxSpec(A) \to \bigsqcup_{\sm \in \MaxSpec(A)} \mathtt B_{\sm}
\quad \textup{given by} \quad
\chi_V(\sm) = 
\begin{cases}
\left[ V_i^{\sm_i^{p_i}} \right], & \textup{if } \sm = \sm_i, \ i \in \{1, \dotsc, k\}, \\
\left[ \C^\sm \right], & \textup{if } \sm \notin \{\sm_1, \dotsc, \sm_k\}.
\end{cases}
\]

The next result shows that two finite-dimensional irreducible $\fg[A]$-modules lie in the same block if and only if their spectral characters are the same.

\begin{prop} \label{prop:blk.map}
Let $V_1, V_2$ be finite-dimensional irreducible $\fg[A]$-modules.  The spectral characters $\chi_{V_1}$ and $\chi_{V_2}$ are the same if and only if $V_1$ and $V_2$ lie in the same block of the category of finite-dimensional $\fg[A]$-modules.
\end{prop}
\begin{proof}
This proof follows from the arguments used in \cite[Lemma~5.18]{NS15}, by using \cite[Proposition~5.5]{CM19} instead of \cite[Theorem~3.7]{NS15}.
\end{proof}

Now that we have described the block decomposition of the category of finite dimensional $\fg[A]$-modules, we turn our attention to the affine Lie superalgebra $\hat \fg = \fg[\C[t, t^{-1}]] \oplus \C c$ (see equation~\ref{eq:aff.sup}). Recall that $c$ acts trivially on every finite-dimensional irreducible $\hat \fg$-module.  Hence, the classifications of finite-dimensional irreducible $\fg[\C[t, t^{-1}]]$- and $\hat\fg$-modules are essentially (that is, bar an inflation) the same. The following result shows that the block decomposition of the category of finite-dimensional modules over $\hat \fg$ is determined by the block decomposition of the category of finite-dimensional modules over $\fg[\C[t, t^{-1}]]$.

\begin{prop} \label{prop:blk.aff}
Let $M$ and $N$ be irreducible finite-dimensional $\hat \fg$-modules.  If $M \not\cong N$, then $\Ext^1_{\hat\fg}(M,N) \cong \Ext^1_{\fg[\C[t, t^{-1}]]}(M,N)$.  In particular, $M$ and $N$ are in the same block in the category of finite-dimensional $\hat\fg$-modules if and only if they are in the same block in the category of finite-dimensional $\fg[\C[t, t^{-1}]]$-modules.
\end{prop}
\begin{proof}
Consider the Lyndon-Hochschild-Serre spectral sequence 
\[
E_2^{p,q}
\cong \Ext^p_{\fg[\C[t, t^{-1}]]} (M, \, \HH^q(\C c, N))
\Rightarrow \Ext^{p+q}_{\hat \fg} (M,\, N).
\]
From its 5-term exact sequence, we extract the following exact sequence:
    \[
0 \to 
\Ext_{\fg[\C[t,t^{-1}]]}^1 (M,\, N) \to
\Ext_{\hat \fg}^1 (M,\, N) \to \HH^1(\C c,\, M^*\otimes N)^{\fg[\C[t,t^{-1}]]}.
    \]
Since $M$ and $N$ are trivial $\C c$-modules, then 
    \[
    \HH^1(\C c,\, M^*\otimes N)^{\fg[\C[t,t^{-1}]]}
    \cong \Hom_{\fg[\C[t,t^{-1}]]}(\C c,\, M^*\otimes N)
    \cong \Hom_{\fg[\C[t,t^{-1}]]}(M,\, N).
    \]
Moreover, since $M$ and $N$ are non-isomorphic irreducible $\fg[\C[t,t^{-1}]]$-modules, it follows that $\Hom_{\fg[\C[t,t^{-1}]]}(M,\, N) = 0$.  The result of the first statement follows.

To prove the \textit{in particular} part, notice that  $M$ and $N$ are isomorphic as $\hat \fg$-modules if and only if they are isomorphic as $\fg[\C[t,t^{-1}]]$-modules. Furthermore, it follows from the first part that if $M\ncong N$, then $\Ext^1_{\hat\fg}(M,N) \cong \Ext^1_{\fg[\C[t, t^{-1}]]}(M,N)$.  Hence $M$ and $N$ are in the same block of the category of finite-dimensional $\hat\fg$-modules if and only if they are in the same block of the category of finite-dimensional $\fg[\C[t, t^{-1}]]$-modules.
\end{proof}

\subsection{Change of Borel subalgebras}

Given an arbitrary Borel subalgebra $\fb = \fh \oplus \fn^+ \subseteq \fg$, every finite-dimensional irreducible $\fg[A]$-module is a $\fb[A]$-highest weight module with $\fb[A]$-highest weight in $\fh[A]^*$ (see the proof of \cite[Lemma~4.5]{Savage14}, which works for arbitrary $\fb$).  That is, if $V$ is an irreducible finite-dimensional $\fg[A]$-module, then there exist $\psi \in \fh[A]^*$ and a non-zero vector $v_\psi \in V$ such that $\fn^+[A] v_\psi = 0$, $H v_\psi = \psi(H)v_\psi$ for all $H \in \fh[A]$, and $V = \bU(\fg[A])v_\psi$.  In this case, $V$ will be denoted by $V_{\fb[A]} (\psi)$ (up to isomorphism, this is well-defined).

Given a set of simple roots $\Delta \subseteq R$ and an odd isotropic simple root, that is, $\alpha \in \Delta_{\bar1}$ such that $\alpha(h_\alpha) = 0$, we define the odd reflection $r_\alpha : \Delta \to R$ to be
    \[
r_\alpha (\beta) := 
     \begin{cases}
       -\alpha &\quad\text{if } \beta=\alpha, \\
       \beta + \alpha &\quad\text{if } \alpha\neq \beta\text{ and } \alpha(h_\beta)\neq 0 \text{ or } \beta(h_\alpha)\neq 0, \\
       \beta &\quad\text{if } \alpha\neq \beta\text{ and } \alpha(h_\beta)=\beta(h_\alpha)=0.
     \end{cases}
    \]
It is known $r_\alpha (\Delta)$ is also a set of simple roots of $\fg$, and that every two such sets can be obtained one from one another by a chain of even and odd reflections (see \cite[Corollary~4.5]{Ser11}).

Let $\fb \subseteq \fg$ be a distinguished Borel subalgebra, $\fb = \fh \oplus \fn^+_0 \oplus \fg_1$, and $\Delta$ be the set of simple roots corresponding to it.  Given a chain of odd isotropic simple roots, $\alpha_1 \in \Delta$, $\alpha_2 \in r_{\alpha_1}(\Delta)$, $\dotsc$, $\alpha_l \in r_{\alpha_{l-1}} \circ \dotsb \circ r_{\alpha_1}(\Delta)$, denote $\sigma := r_l \circ \dotsb \circ r_{\alpha_1}$ and denote by $\fb_\sigma$ the Borel subalgebra of $\fg$ corresponding to $\sigma(\Delta)$.  Since every two sets of simple roots of $\fg$ can be obtained one from one another by a chain of even and odd reflections, every Borel subalgebra of $\fg$ is equal to $\fb_\sigma$.  In this section, our first goal is to establish a relation between the highest weights of a finite-dimensional irreducible $\fg[A]$-module relative to Borel subalgebras that are not conjugate under the action of the Weyl group of $\fg_0$.

Given a finite-dimensional irreducible $\fg[A]$-module $V_{\fb[A]}(\psi)$, define the ideal $\Ann_A V_{\fb[A]}(\psi) := \{ a \in A \mid \psi(h \otimes a) = 0 \textup{ for all } h \in \fh\}\subseteq A$ and the support of $V_{\fb[A]}(\psi)$ to be the set $\Supp V_{\fb[A]}(\psi) = \{\sm\in \MaxSpec A \mid \Ann_A V_{\fb[A]}(\psi)\subseteq \sm\}$ (cf. \cite[Definition~3.5, Lemma~3.6 and Lemma~4.5]{Savage14}).  We will begin by analysing the case where $\Supp V_{\fb[A]}(\psi)$ consists of a unique maximal ideal $\sm$.  

Given $\lambda\in \fh^*$ and a maximal ideal $\sm \subseteq A$, recall that the evaluation map $\ev_\sm : A \to \C$ is induced by the natural projection $a \mapsto a + \sm$, and define $\lambda \otimes \ev_\sm$ to be the unique linear functional in $\fh[t]^*$ that satisfies $(\lambda \otimes \ev_\sm) (h \otimes a) = \ev_\sm(a) \lambda(h)$ for all $h \in \fh$ and $a \in A$.  Further, notice that if $V$ is a finite-dimensional irreducible $\fg[A]$-module such that $\Supp(V) = \{\sm\}$, then $V \cong \bV_A (\Theta \boxtimes \sV)$ where $\Theta = \left. \psi \right|_{\fz[A]}$ and $\kteta \subseteq \sm$.
\details{
In fact, notice that, since $\fz[A] \subseteq \fh[A]$, then $\Ann_A(V) \subseteq \kteta$.  Now, let $\mathsf n$ be a maximal ideal containing $\kteta$.  Since $\Ann_A(V) \subseteq \kteta \subseteq \mathsf n$ and $\Supp(V) = \{\sm\}$, then $\mathsf n = \sm$.  That is, $\sm$ is the unique maximal ideal that contains $\kteta$.
}
In this case, denote by $d_{\psi, \sm} := \dim (A/\kteta)$.

\begin{lem} \label{lem:chang.Borel}
Let $\fb = \fh \oplus \fn^+_0 \oplus \fg_1$ be a distinguished Borel subalgebra of $\fg$, $\Delta$ be the set of simple roots associated to $\fb$, $\psi \in \fh[A]^*$ be such that $V_{\fb[A]}(\psi)$ is a finite-dimensional irreducible $\fg[A]$-module, and $\sigma = r_{\alpha_l} \circ \cdots \circ r_{\alpha_1}$ be a chain of odd reflections on $\Delta$. If $\Supp V_{\fb[A]}(\psi) = \{ \sm \}$, then there is an isomorphism of $\fg[A]$-modules     
    \[
V_{\fb[A]}(\psi) 
\cong V_{\fb_\sigma[A]} \left( \psi - d_{\psi, \sm}(\alpha_1 + \dotsb + \alpha_l) \otimes \ev_\sm \right).
    \]
\end{lem}
\begin{proof}
Throughout this proof, set $\Theta := \left. \psi \right|_{\fz[A]}$. 

Assume first that $d_{\psi, \sm} = 1$.  In this case $\kteta = \sm$, $V_{\fb[A]}(\psi)$ is an evaluation module, and $\psi = \lambda \otimes \ev_\sm$ for $\lambda = \left. \psi \right|_{\fh}$.  Thus, the fact that $V_{\fb[A]}(\psi) \cong V_{\fb_\alpha[A]} \left( \psi - d_{\psi, \sm}(\alpha_1 + \dotsb + \alpha_l) \otimes \ev_\sm \right)$ follows from the fact that $\psi - d_{\psi, \sm}(\alpha_1 + \dotsb + \alpha_l) \otimes \ev_\sm = (\lambda - \alpha_1 - \dotsb - \alpha_l) \otimes \ev_\sm$ and \cite[Lemma~10.2]{Ser11}.

Assume now that $d_{\psi, \sm} > 1$, or equivalently, that $\kteta \subsetneq \sm$.  Thus, in this case, it follows from Theorem~\ref{theo:class.irreps} that $V_{\fb[A]}(\psi) \cong \bK_{A/\kteta} (\Theta \boxtimes \sV^\sm)$, where $\sV$ is the unique finite-dimensional irreducible $\fg_0'$-module of highest weight $\left. \psi \right|_{\fh'}$.  In order to show that $\bK_{A/\kteta}(\Theta \boxtimes \sV^\sm)$ is isomorphic to $V_{\fb_\sigma[A]} \left( \psi - d_{\psi, \sm}(\alpha_1 + \dotsb + \alpha_l) \otimes \ev_\sm \right)$, we will construct a highest-weight vector, with respect to $\fb_\sigma[A]$, of weight $\psi - d_{\psi, \sm}(\alpha_1 + \dotsb + \alpha_l) \otimes \ev_\sm$ inside $\bK_{A/\kteta}(\Theta \boxtimes \sV^\sm)$.  In order to do that, fix an ordered basis $\{ \bar{a_1}, \dotsc, \bar{a_n}\}$ for $A/\kteta$ such that $b {\bar a}_i \in {\rm span}_\C \{\bar{a_{i+1}}, \dotsc, \bar{a_n} \}$ and $b \bar{a_n} = 0$ for all $b \in \sm$ and $i \in \{1, \dotsc, n-1\}$.
\details{
In order to construct one such basis, consider the filtration on $A/\kteta$ induced by the powers of the maximal ideal $\sm$, namely, $\dotsb \subseteq \sm^i/\kteta \subseteq \dotsb \subseteq \sm/\kteta \subseteq A/\kteta$.  The graded algebra associated to this filtration is defined to be ${\rm Gr}(A/\kteta) := \bigoplus_{i \ge 0} \sm^i / \sm^{i+1}$.  Since $\Supp V_{\fb[A]}(\psi) = \{\sm\}$, then there exists $p \ge 0$ such that $\sm^p \subseteq \kteta$; and thus, $\sm^q/\kteta = 0$ for all $q \ge p$.  Now, choose a basis $\mathsf B$ for ${\rm Gr}(A)$ by choosing a basis for each vector space $\sm^i/\sm^{i+1}$, $i \in \{0, \dotsc, p-1\}$.  Then, for each $b \in \mathsf B$, choose a representative in $A/\kteta$.  By construction, the resulting set will be a basis for $A/\kteta$ with the desired property.
}
Now, let $k_\psi$ be a non-zero vector in $\bK_{A/\kteta}(\Theta \boxtimes \sV^\sm)$ such that $\fn^+[A] k_\psi = 0$, $H k_\psi = \psi(H)k_\psi$ for all $H \in \fh[A]$, and $\bK_{A/\kteta}(\Theta \boxtimes \sV^\sm) = \bU(\fg[A])k_\psi$.  Also let $Y := (y_{\alpha_l} \otimes \overline{a_1}) \dotsm (y_{\alpha_l} \otimes \overline{a_n}) \dotsm (y_{\alpha_1} \otimes \overline{a_1}) \dotsm (y_{\alpha_1} \otimes \overline{a_n})$ and $u_\psi := Y k_\psi$.  Notice that $u_\psi$ is nonzero and that
    \[
    (x_\beta \otimes a) u_\psi = (y_{\alpha_i} \otimes a) u_\psi = 0, \quad
    (h \otimes b) u_\psi = \psi(h \otimes b) u_\psi \quad \textup{and} \quad
    h u_\psi = (\lambda - d_{\psi, \sm} (\alpha_1 + \dotsb + \alpha_l))(h) u_\psi,
    \]
for all $\beta \in \Delta \setminus \{\alpha_1, \dotsc, \alpha_l\}$, $i \in \{1, \dotsc, l\}$, $a \in A$, $h \in \fh$ and $b \in \sm$.
\details{
In order to show that $u_\psi$ is nonzero, denote $\dim \fg_{-1}[A/\kteta] =: d$ and  notice that $Y$ is a sub-expression of a nonzero element in $\bLambda^d (\fg_{-1}[A/\kteta])$.  In fact, since $d = \dim \fg_{-1}[A/\kteta]$, every nonzero element in $\bLambda^d (\fg_{-1}[A/\kteta])$ is a scalar multiple of $p := (y_{\beta_1} \otimes \overline{a_1}) \dotsm (y_{\beta_1} \otimes \overline{a_n}) \dotsm (y_{\beta_\ell} \otimes \overline{a_1}) \dotsm (y_{\beta_\ell} \otimes \overline{a_n})$, where $\{ \beta_1, \dotsc, \beta_\ell \} = \Delta$.  Hence, $p = qY$, where $q = \pm (y_{\gamma_1} \otimes \overline{a_1}) \dotsm (y_{\gamma_1} \otimes \overline{a_n}) \dotsm (y_{\gamma_m} \otimes \overline{a_1}) \dotsm (y_{\gamma_m} \otimes \overline{a_n})$, where $\{\gamma_1, \dotsc, \gamma_m \} = \Delta \setminus \{ \alpha_1, \dotsc, \alpha_l \}$. Thus, if follows from successively applying Lemma~\ref{lem:comm.rels} that $p^\star q u_\psi = p^\star p k_\psi$ is a nonzero scalar multiple of $k_\psi$ (compare with the last paragraph of the proof of Proposition~\ref{prop:main}).  This shows that $u_\psi$ is nonzero.

In order to show that $(x_\beta \otimes a) u_\psi = 0$ for all $a \in A$, first notice that
\begin{align*}
(x_\beta & \otimes a) u_\psi \\
&= (x_\beta \otimes a) (y_{\alpha_l} \otimes \overline{a_1}) \dotsm (y_{\alpha_l} \otimes \overline{a_n}) \dotsm (y_{\alpha_1} \otimes \overline{a_1}) \dotsm (y_{\alpha_1} \otimes \overline{a_n}) k_\psi \\
&= \sum_{i=1}^l \sum_{j=1}^n (y_{\alpha_l} \otimes \overline{a_1}) \dotsm (y_{\alpha_l} \otimes \overline{a_n}) \dotsm \widehat{(y_{\alpha_i} \otimes \overline{a_j})} \left( [x_\beta, y_{\alpha_i}] \otimes \overline{a a_j} \right) \dotsm (y_{\alpha_1} \otimes \overline{a_1}) \dotsm (y_{\alpha_1} \otimes \overline{a_n}) k_\psi.
\end{align*}
Then notice that $[x_\beta, y_{\alpha_i}] = 0$ for all $\beta \in \sigma(\Delta)$, since $\beta$ and $\alpha_i$ are simple roots and thus $\beta - \alpha_i \notin R$.  

In order to show that $(y_{\alpha_i} \otimes a) u_\psi = 0$ for all $i \in \{1, \dotsc, l\}$ and $a \in A$, first notice that
\[
(y_{\alpha_i} \otimes a) u_\psi
= (y_{\alpha_i} \otimes a) (y_{\alpha_l} \otimes \overline{a_1}) \dotsm (y_{\alpha_l} \otimes \overline{a_n}) \dotsm (y_{\alpha_1} \otimes \overline{a_1}) \dotsm (y_{\alpha_1} \otimes \overline{a_n}) k_\psi.
\]
Then notice that, since $\alpha_i \in \{\alpha_1, \dotsc, \alpha_l\}$ and $\{ \overline{a_1}, \dotsc, \overline{a_n} \}$ is a basis of $A/\kteta$, then, for each $i \in \{1, \dotsc, l\}$, there exist $\mu_{i,1}, \dotsc, \mu_{i,n} \in \C$ such that $y_{\alpha_i} \otimes a = \sum_{j=1}^n \mu_{i,j} (y_{\alpha_i} \otimes \overline{a_j})$.  Since $y_{\alpha_i} \otimes \overline{a_j}$ is a factor of $Y$ for all $i \in \{1, \dotsc, n\}$ and $j \in \{1, \dotsc, n\}$, it follows that $Y (y_{\alpha_i} \otimes a) = 0$.

In order to show that $(h \otimes b) u_\psi = \psi(h \otimes b) u_\psi$ for all $h \in \fh$ and $b \in \sm$, first notice that 
\begin{align*}
(h & \otimes b) u_\psi \\
&= (h \otimes b) (y_{\alpha_l} \otimes \overline{a_1}) \dotsm (y_{\alpha_l} \otimes \overline{a_n}) \dotsm (y_{\alpha_1} \otimes \overline{a_1}) \dotsm (y_{\alpha_1} \otimes \overline{a_n}) k_\psi \\
&= \sum_{i=1}^l \sum_{j=1}^n (y_{\alpha_l} \otimes \overline{a_1}) \dotsm (y_{\alpha_l} \otimes \overline{a_n}) \dotsm \widehat{(y_{\alpha_i} \otimes \overline{a_j})} \left( [h, y_{\alpha_i}] \otimes \overline{b a_j} \right) \dotsm (y_{\alpha_1} \otimes \overline{a_1}) \dotsm (y_{\alpha_1} \otimes \overline{a_n}) k_\psi \\
&{\quad}+ Y (h \otimes b) k_\psi \\
&= \sum_{i=1}^l \sum_{j=1}^n (y_{\alpha_l} \otimes \overline{a_1}) \dotsm (y_{\alpha_l} \otimes \overline{a_n}) \dotsm \widehat{(y_{\alpha_i} \otimes \overline{a_j})} \left( -\alpha_i(h) y_{\alpha_i} \otimes \overline{b a_j} \right) \dotsm (y_{\alpha_1} \otimes \overline{a_1}) \dotsm (y_{\alpha_1} \otimes \overline{a_n}) k_\psi \\
&{\quad}+ \psi(h \otimes b) u_\psi.
\end{align*}
Then notice that, since $b \in \sm$ and $\{ \overline{a_1}, \dotsc, \overline{a_n} \}$ is a basis of $A/\kteta$ such that $b \overline{a_j} \in {\rm span}_\C \{\bar{a_{j+1}}, \dotsc, \bar{a_n} \}$ for all $j \in \{1, \dotsc, n\}$, then, for each $i \in \{1, \dotsc, l\}$ and $j \in \{1, \dotsc, n\}$, there exist $\mu_{i,j+1}, \dotsc, \mu_{i,n} \in \C$ such that $y_{\alpha_i} \otimes \overline{b a_j} = \sum_{k=j+1}^n \mu_{i,k} (y_{\alpha_i} \otimes \overline{a_k})$.  Since $y_{\alpha_i} \otimes \overline{a_k}$ is a factor of $Y$ distinct from $y_{\alpha_i} \otimes \overline{a_j}$, then $(y_{\alpha_l} \otimes \overline{a_1}) \dotsm (y_{\alpha_l} \otimes \overline{a_n}) \dotsm \widehat{(y_{\alpha_i} \otimes \overline{a_j})} \left( y_{\alpha_i} \otimes \overline{b a_j} \right) \dotsm (y_{\alpha_1} \otimes \overline{a_1}) \dotsm (y_{\alpha_1} \otimes \overline{a_n}) = 0$  for all $i \in \{1, \dotsc, l\}$, $j \in \{1, \dotsc, n\}$ and $k \in \{j+1, \dotsc, n\}$.

Finally, in order to show that $h u_\psi = (\lambda - d_{\psi, \sm}(\alpha_1 + \dotsb + \alpha_n))(h) u_\psi$ for all $h \in \fh$, notice that 
\begin{align*}
h u_\psi
&= h (y_{\alpha_l} \otimes \overline{a_1}) \dotsm (y_{\alpha_l} \otimes \overline{a_n}) \dotsm (y_{\alpha_1} \otimes \overline{a_1}) \dotsm (y_{\alpha_1} \otimes \overline{a_n}) k_\psi \\
&= \sum_{i=1}^l \sum_{j=1}^n (y_{\alpha_l} \otimes \overline{a_1}) \dotsm (y_{\alpha_l} \otimes \overline{a_n}) \dotsm \widehat{(y_{\alpha_i} \otimes \overline{a_j})} \left( [h, y_{\alpha_i}] \otimes \overline{a_j} \right) \dotsm (y_{\alpha_1} \otimes \overline{a_1}) \dotsm (y_{\alpha_1} \otimes \overline{a_n}) k_\psi \\
&{\quad}+ Y h k_\psi \\
&= \sum_{i=1}^l \sum_{j=1}^n (y_{\alpha_l} \otimes \overline{a_1}) \dotsm (y_{\alpha_l} \otimes \overline{a_n}) \dotsm \widehat{(y_{\alpha_i} \otimes \overline{a_j})} \left( -\alpha_i(h) y_{\alpha_i} \otimes \overline{a_j} \right) \dotsm (y_{\alpha_1} \otimes \overline{a_1}) \dotsm (y_{\alpha_1} \otimes \overline{a_n}) k_\psi \\
&{\quad}+ \lambda(h) u_\psi \\
&= (\lambda - d_{\psi, \sm}(\alpha_1 + \dotsb + \alpha_n))(h) u_\psi.
\end{align*}
}
This shows that, with respect to $\fb_\sigma [A]$, $u_\psi$ is a highest-weight vector of weight $\psi - d_{\psi, \sm}(\alpha_1 + \dotsb + \alpha_l) \otimes \ev_\sm$. Since $\bK_{A/\kteta}(\Theta \boxtimes \sV^\sm)$ is assumed to be irreducible, it follows from \cite[Lemma~4.5]{Savage14} that $\bK_{A/\kteta}(\Theta \boxtimes \sV^\sm)$ is isomorphic to $V_{\fb_\sigma[A]} \left( \psi - d_{\psi, \sm}(\alpha_1 + \dotsb + \alpha_l) \otimes \ev_\sm \right)$ as a $\fg[A]$-module.
\end{proof}

\begin{rem}
In the case where $V_{\fb[A]}(\psi)$ is the trivial $\fg[A]$-module, then $V_{\fb_\sigma[A]}(\psi) \cong V_{\fb[A]}(\psi) \cong \C$.
\end{rem}

We now use Lemma~\ref{lem:chang.Borel} to establish a relation between the highest weights of a general fi\-ni\-te\--dimen\-sion\-al irreducible $\fg[A]$-module relative to distinct Borel subalgebras.

\begin{prop} \label{prop:chang.Borel}
Let $\fb = \fh \oplus \fn^+_0 \oplus \fg_1$ be a distinguished Borel subalgebra of $\fg$, $\psi \in \fh[A]^*$ be such that $V_{\fb[A]}(\psi)$ is a finite-dimensional irreducible $\fg[A]$-module, and $\sigma = r_{\alpha_l} \circ \cdots \circ r_{\alpha_1}$ be a chain of odd reflections.  If the support of $V_{\fb[A]}(\psi)$ is $\{ \sm_1, \dotsc, \sm_n \}$, then there is an isomorphism of $\fg[A]$-modules
\[
V_{\fb[A]}(\psi) 
\cong V_{\fb_\sigma[A]} \left( \psi - d_{\psi, \sm_1} (\alpha_1 + \dotsb + \alpha_l) \otimes \ev_{\sm_1} - \dotsb - d_{\psi, \sm_n} (\alpha_1 + \dotsb + \alpha_l) \otimes \ev_{\sm_n} \right).
\]
\end{prop}
\begin{proof}
First recall that, since $A$ is a finitely-generated $\C$-algebra, the support of any fi\-ni\-te\--dimen\-sion\-al irreducible $\fg[A]$-module is finite (see \cite[Theorem~4.16]{Savage14}).  Hence, there exist pairwise-distinct maximal ideals $\sm_1, \dotsc, \sm_n \subseteq A$ such that $\Supp V_{\fb[A]}(\psi) = \{\sm_1, \dotsc, \sm_n\}$.  Moreover, there exist $\psi_1, \dotsc, \psi_n \in \fh[A]^*$ such that $\Supp V_{\fb[A]}(\psi_i) = \{\sm_i\}$ for all $i \in \{1, \dotsc, n\}$ and $V_{\fb[A]}(\psi) = V_{\fb[A]}(\psi_1) \otimes \dotsm \otimes V_{\fb[A]}(\psi_n)$ (see \cite[Corollary~4.14]{Savage14}). Now the statement follows from Lemma~\ref{lem:chang.Borel}.
\end{proof}


\end{document}